\documentclass[final]{dmtcs-episciences}

\usepackage{amsfonts}
\usepackage{amsmath}
\usepackage{amsthm}
\usepackage{stmaryrd}
\usepackage[utf8]{inputenc}
\usepackage{subfigure}
\usepackage[round]{natbib}

\newcommand{\lasm}{LASM}
\newcommand{\sasm}{SASM}

\newtheorem{theorem}{Theorem}[section]

\newtheorem{example}[theorem]{Example}
\newtheorem{definition}[theorem]{Definition}

\newtheorem{proposition}[theorem]{Proposition}

\newtheorem{lemma}[theorem]{Lemma}
\newtheorem{corollary}[theorem]{Corollary}

\newtheorem{conjecture}[theorem]{Conjecture}

\author{Eric S. Egge\affiliationmark{1}\thanks{Corresponding author.} \and Kailee Rubin\affiliationmark{2}}

\affiliation{
Carleton College, Northfield, MN, USA\\
Epic, Madison, WI, USA}

\title{Snow Leopard Permutations and their Even and Odd Threads}


\received{2015-08-24}
\revised{2016-05-03}
\accepted{2016-05-12}

\begin{document}

\publicationdetails{18}{2016}{2}{5}{1279}

\maketitle

\begin{abstract}
Caffrey, Egge, Michel, Rubin and Ver Steegh recently introduced snow leopard permutations, which are the anti-Baxter permutations that are compatible with the doubly alternating Baxter permutations.
Among other things, they showed that these permutations preserve parity, and that the number of snow leopard permutations of length $2n-1$ is the Catalan number $C_n$.
In this paper we investigate the permutations that the snow leopard permutations induce on their even and odd entries;  we call these the {\em even threads} and the {\em odd threads}, respectively.
We give recursive bijections between these permutations and certain families of Catalan paths.
We characterize the odd (resp.~even) threads which form the other half of a snow leopard permutation whose even (resp.~odd) thread is layered in terms of pattern avoidance, and we give a constructive bijection between the set of permutations of length $n$ which are both even threads and odd threads and the set of peakless Motzkin paths of length $n+1$.
\end{abstract}

\section{Introduction}
\label{sec:intro}

A {\em complete Baxter permutation} $\pi$, as introduced by \cite{BaxterFixed} and characterized by \cite{NumberBaxterGraham}, is a permutation of length $2n+1$ such that, for all $i$ with $1 \leq i \leq 2n+1$,
\begin{itemize}
\item $\pi(i)$ is even if and only if $i$ is even and
\item if $\pi(x) = i$, $\pi(y) = i+1$, and $z$ is between $x$ and $y$ (that is, $x < z < y$ or $y < z < x$), then $\pi(z) < i$ if $i$ is odd and $\pi(z) > i+1$ if $i$ is even.
\end{itemize}
As Chung and her coauthors note, each complete Baxter permutation is uniquely determined by its odd entries.  
Accordingly, for each complete Baxter permutation $\pi$ of length $2n+1$, the associated {\em reduced Baxter permutation} of length $n+1$ is the permutation $\pi$ induces on its odd entries.
For example, 981254367 is a complete Baxter permutation of length nine, whose associated reduced Baxter permutation is 51324.
One can show that a permutation is a reduced Baxter permutation if and only if it avoids the vincular patterns $3\underline{14}2$ and $2\underline{41}3$.  
That is, if $\pi$ has length $n$, then $\pi$ is a reduced Baxter permutation whenever there are no indices $1 \leq i<j<j+1<k<n$ such that $\pi(j)<\pi(k)<\pi(i)<\pi(j+1)$ (for $3\underline{14}2$) or $\pi(j+1)<\pi(i)<\pi(k)<\pi(j)$ (for $2\underline{41}3$).  
For example, the complete Baxter permutation $3\ 2\ 1\ 4\ 13\ 12\ 7\ 8\ 11\ 10\ 9\ 6\ 5$ has reduced Baxter permutation $2174653$, which has no instances of $3\underline{14}2$ or $2\underline{41}3$.  
In contrast, $4613752$ is not a reduced Baxter permutation because the subsequence $6375$ is an instance of $3\underline{14}2$.
We sometimes refer to reduced Baxter permutations as {\em Baxter permutations} for short.

For each complete Baxter permutation $\pi$ of length $2n+1$, the associated {\em anti-Baxter permutation} is the permutation of length $n$ that $\pi$ induces on its even entries.
Although complete Baxter permutations are determined by their associated Baxter permutations, they are not determined by their associated anti-Baxter permutations:  many anti-Baxter permutations appear in several complete Baxter permutations.
For example, 4123 is the anti-Baxter permutation associated with both of the complete Baxter permutations 981254367 and 983214765.
On the other hand, anti-Baxter permutations are characterized by pattern avoidance conditions nearly identical to those which characterize the Baxter permutations:  a permutation $\pi$ is anti-Baxter if and only if it avoids the vincular patterns $3\underline{41}2$ and $2\underline{14}3$.  
That is, if $\pi$ has length $n$, then $\pi$ is anti-Baxter whenever there are no indices $1 \leq i<j<j+1<k<n$ such that $\pi(j+1)<\pi(k)<\pi(i)<\pi(j)$ (for $3\underline{41}2$) or $\pi(j)<\pi(i)<\pi(k)<\pi(j+1)$ (for $2\underline{14}3$).
We say a Baxter permutation $\pi_1$ and an anti-Baxter permutation $\pi_2$ are \emph{compatible} whenever there is a complete Baxter permutation $\pi$ such that $\pi_1$ and $\pi_2$ are the permutations induced on the odd and even entries of $\pi$, respectively.  
One can show that this definition of compatibility is equivalent to the definition of compatibility given by \cite{Comps}.   

Baxter first introduced the permutations that now bear his name in connection with a problem involving fixed points of commuting continuous functions on the closed interval $[0,1]$, but they have also appeared in a variety of other settings.
One such setting involves the {\em Aztec diamond} of order $n$, which is an array of unit squares with $2i$ squares in row $i$  for $1 \leq i \leq n$ and $2(2n-i+1)$ squares in row $i$ for $n < i \leq 2n$, in which the squares are centered in each row.
In Figure \ref{ADexample} we have the Aztec diamond of order three.
\begin{figure}[ht]
\centering
\includegraphics[width=4cm]{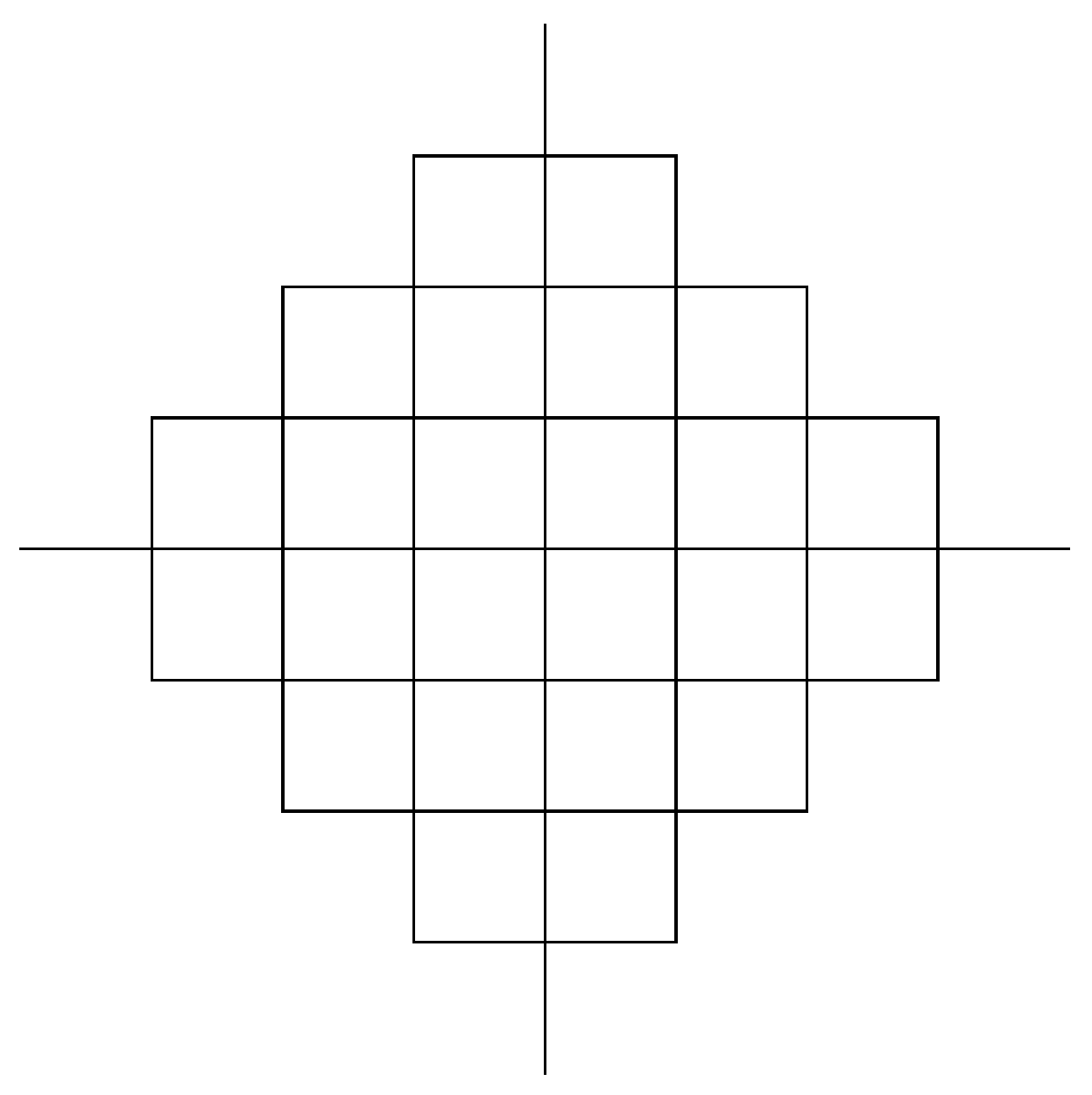}
\caption{The Aztec diamond of order 3.}
\label{ADexample}
\end{figure}
The entries of our complete Baxter permutations will be at the vertices of an Aztec diamond.
We prefer to arrange these vertices in rows and columns, so we will orient all of our Aztec diamonds as in Figure \ref{fig:ADrotated}.
\begin{figure}[ht]
\centering
\setlength{\unitlength}{.15in}
\begin{picture}(8,8)
\put(0,0){\line(1,1){8}}
\put(0,2){\line(1,1){6}}
\put(0,4){\line(1,1){4}}
\put(0,6){\line(1,1){2}}
\put(2,0){\line(1,1){6}}
\put(4,0){\line(1,1){4}}
\put(6,0){\line(1,1){2}}
\put(2,0){\line(-1,1){2}}
\put(4,0){\line(-1,1){4}}
\put(6,0){\line(-1,1){6}}
\put(8,0){\line(-1,1){8}}
\put(8,2){\line(-1,1){6}}
\put(8,4){\line(-1,1){4}}
\put(8,6){\line(-1,1){2}}
\end{picture}
\caption{The Aztec diamond of order 3, reoriented.}
\label{fig:ADrotated}
\end{figure}
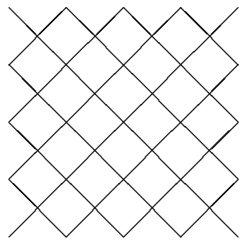

Aztec diamonds can be tiled by dominos, and \cite{EKLP} have described how to construct, for each such tiling of the Aztec diamond of order $n$, a pair of matrices $\sasm(T)$ and $\lasm(T)$ of sizes $n\times n$ and $(n+1) \times (n+1)$, respectively.
Each of these matrices is an {\em alternating-sign matrix} (ASM), which is a matrix with entries in $\{0,1,-1\}$ whose nonzero entries in each row and in each column alternate in sign and sum to one.
(Note that SASM and LASM above are short for small ASM and large ASM, respectively.
For an introduction to ASMs and a variety of related combinatorial objects, see the work of \cite{RobbinsStory}, \cite{Bressoud}, and \cite{ProppASM}.)
To carry out this construction, first note that a tiling of an Aztec diamond with dominos gives rise to a graph whose vertices are lattice points, as in Figure \ref{onedot}.
\begin{figure}[ht]
\centering
\includegraphics[width=3.8cm]{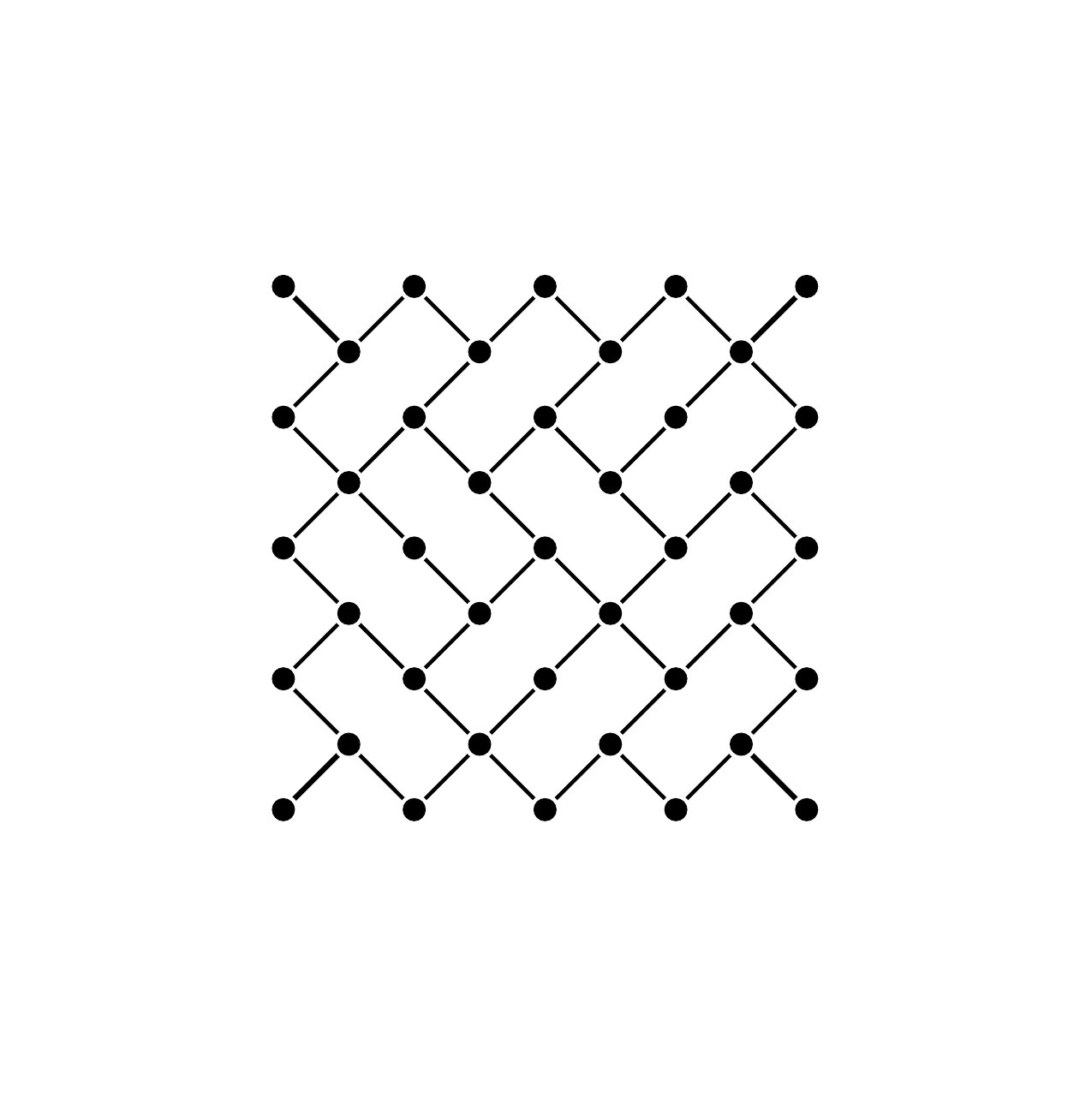}
\caption{The graph of a domino tiling of the Aztec diamond of order 3.}
\label{onedot}
\end{figure}
Each exterior vertex of this graph has the same degree in all domino tilings, so we disregard these vertices.
The remaining vertices fall naturally into two sets as in Figure \ref{twodots}: the black vertices form an $(n+1) \times (n+1)$ matrix while the white vertices form an $n \times n$ matrix. 
\begin{figure}[ht]
\centering
\includegraphics[width=3.8cm]{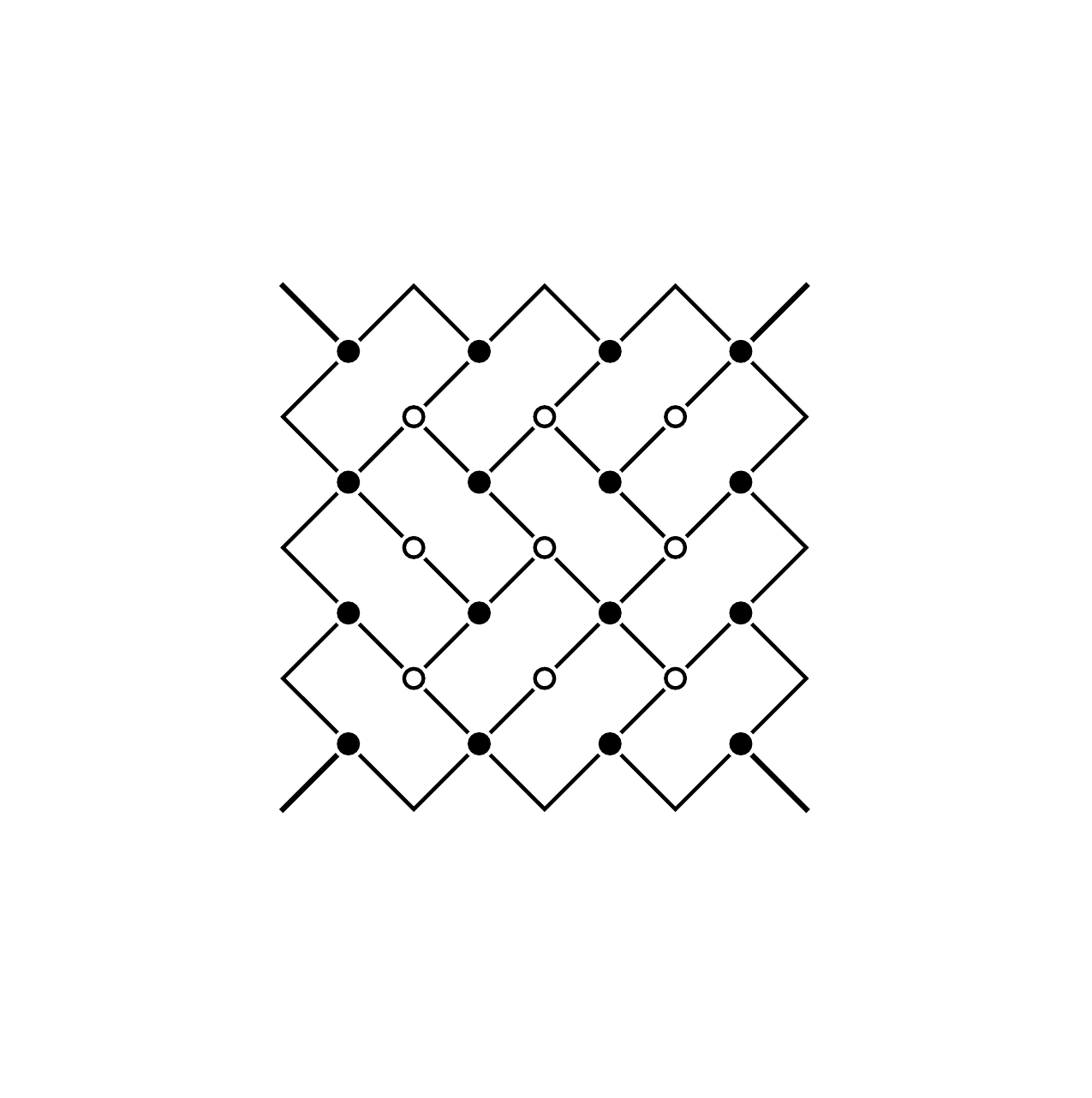}
\caption{The graph of a domino tiling of the Aztec diamond of order 3 with outer vertices removed.}
\label{twodots}
\end{figure}
We construct $\lasm(T)$ on the black vertices by labeling each vertex of degree four with a $1$, labeling each vertex of degree three with a $0$, and labeling each vertex of degree two with a $-1$. 
We construct $\sasm(T)$ on the white vertices by labeling each vertex of degree four with a $-1$, labeling each vertex of degree three with a $0$, and labeling each vertex of degree two with a $1$.  
For example, the tiling $T$ in Figure \ref{twodots} has
$$\lasm(T)= \begin{bmatrix}0 & 0 & 0 & 1 \\ 1 & 0 & 0 & 0 \\ 0 & 0 & 1 & 0 \\ 0 & 1 & 0 & 0\end{bmatrix}$$
and
$$\sasm(T) = \begin{bmatrix}0 & 0 & 1 \\ 1 & 0 & 0 \\ 0 & 1 & 0\end{bmatrix}.$$
\cite{Canary} has shown that $\lasm(T)$ is a permutation matrix if and only if it is the matrix for a Baxter permutation.
In this case, if we now introduce gray dots inside the dominos as in Figure \ref{threedots}, 
\begin{figure}[ht]
\centering
\includegraphics[width=3.8cm]{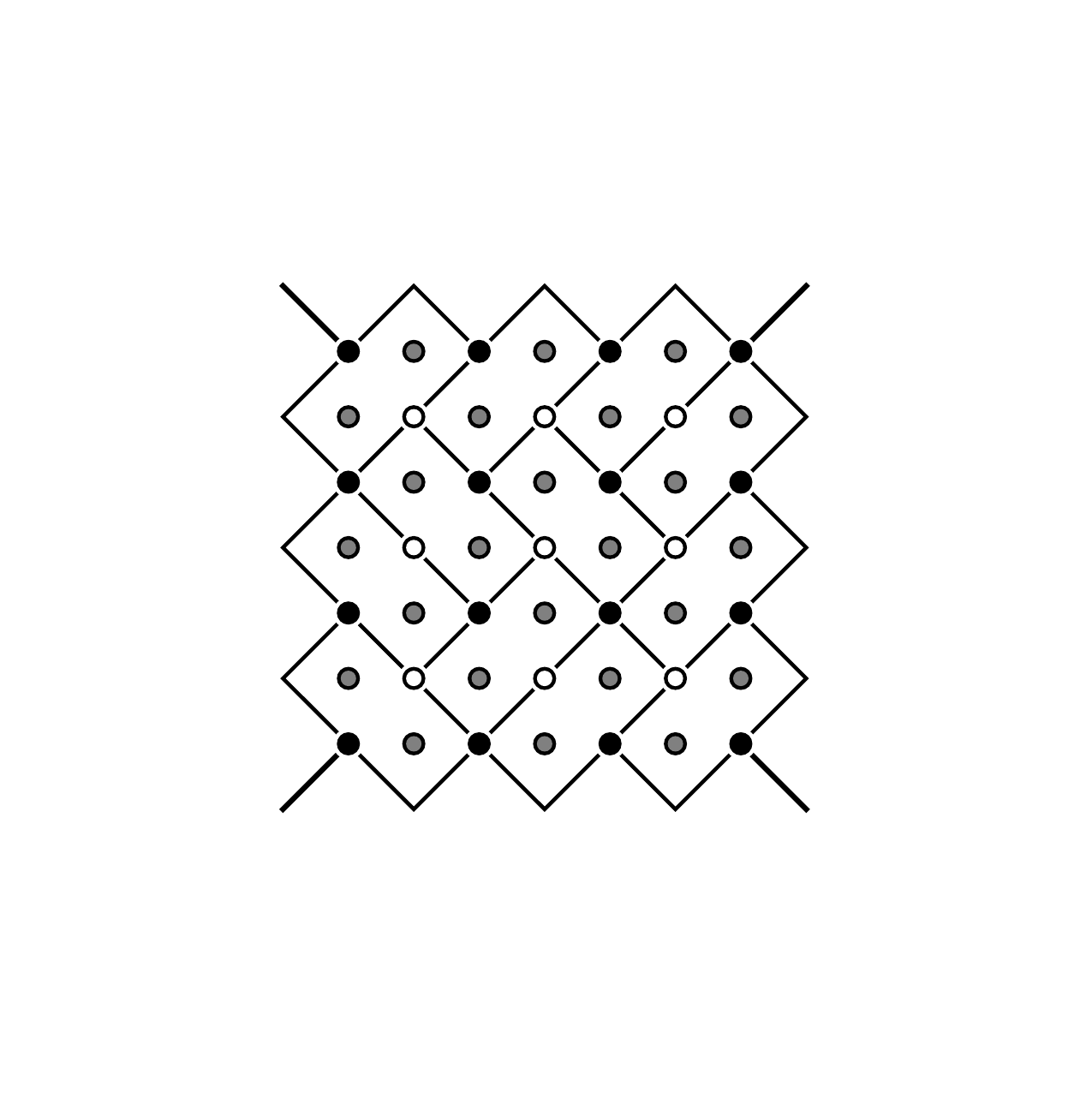}
\caption{The entries of the associated complete Baxter permutation are at the black, white, and gray dots.}
\label{threedots}
\end{figure}
and assign 0s to each of them, then the black, white, and gray dots together form the matrix for a complete Baxter permutation.

While Baxter permutations have been studied in several other combinatorial incarnations (see, for example, the papers of \cite{Floorplans} and \cite{DilksInv}), we focus on a certain subset of the Baxter permutations, and their compatible anti-Baxter permutations.
Specifically, an {\em alternating permutation}, or an {\em up-down permutation}, is a permutation which begins with an ascent, and in which ascents and descents alternate.
Similarly, a {\em doubly alternating permutation} is an alternating permutation whose inverse is also alternating.
\cite{DABPCatalan} have shown that the number of doubly alternating Baxter permutations of length $2n$ is the Catalan number $C_n = \frac{1}{n+1} \binom{2n}{n}$, as is the number of doubly alternating Baxter permutations of length $2n+1$.
Building on this, \cite{Comps} make the following definition.
\begin{definition}
\label{defn:snowleopardpermutation}
A {\em snow leopard permutation} is an anti-Baxter permutation which is compatible with a doubly alternating Baxter permutation.
\end{definition}
\cite{Comps} show that compatibility is a bijection between the set of doubly alternating Baxter permutations of length $n$ and the set of snow leopard permutations of length $n-1$.
They also give a simple bijection between the set of snow leopard permutations of length $2n$ and the set of snow leopard permutations of length $2n-1$ (see Theorem \ref{thm:slpdecom}), so we focus our attention on the snow leopard permutations of odd length.
We write $SL_n$ to denote the set of snow leopard permutations of length $2n-1$;  in Table \ref{table:sln} we list the snow leopard permutations of lengths one, three, and five.
\begin{table}[ht]
\centering
\begin{tabular}{c|c}
$n$ & $SL_n$ \\
\hline
1 & 1 \\
\hline
3 & 123,321 \\
\hline
5 & 12345,14325,34521,54123,54321
\end{tabular}
\caption{The snow leopard permutations of lengths one, three, and five.}
\label{table:sln}
\end{table}

Like the complete Baxter permutations, and as \cite{Comps} show, the snow leopard permutations preserve parity.
In this paper, we study the permutations the snow leopard permutations induce on their odd and even entries;  we call these induced permutations {\em odd threads} and {\em even threads}, respectively.  
In Section \ref{sec:tools} we recall some useful permutation tools.
In Section \ref{sec:evenodd} we give recursive decompositions of the even and odd threads, and we explore the extent to which these decompositions are unique.
In Section \ref{sec:eopaths} we use our decompositions to give recursive bijections between the set of even threads of length $n$ and the set of Catalan paths of length $n+1$ with no ascent of length two, and between the set of odd threads of length $n$ and the set of Catalan paths of length $n$ with no four consecutive steps $NEEN$.
In Section \ref{sec:entangled} we begin to describe which even and odd threads are induced by the same snow leopard permutation.
In particular, we show that the even threads which can be paired with the increasing odd thread to form a snow leopard permutation are exactly the $3412$-avoiding involutions, we show that the odd threads which can be paired with the decreasing even thread to form a snow leopard permutation are exactly the complements of the $3412$-avoiding involutions, and we extend these results in a natural way to layered even and odd threads.
Finally, in Section \ref{sec:janusthreads} we give a constructive bijection between the set of permutations of length $n$ which are both even threads and odd threads and the set of peakless Motzkin paths of length $n+1$.

\section{Permutation Tools}
\label{sec:tools}

Throughout we write $S_n$ denote the set of all permutations of length $n$, written in one-line notation, and for any permutation $\pi$ we write $|\pi|$ to denote the length of $\pi$. 
The complement operation on permutations will be useful for us, so we recall it next.

\begin{definition}
Following \cite{Kitaev}, for any permutation $\pi \in S_n$, we write $c(\pi)$ to denote the complement of $\pi$, which is the permutation in $S_n$ with
$$c(\pi)(j) = n+1-\pi(j)$$
for all $j,$ $1\leq j \leq n$.  
\end{definition}

We will also make extensive use of the following two ways of combining two permutations.

\begin{definition}
For permutations $\pi \in S_n$ and $\sigma \in S_m$ we write $\pi \oplus \sigma$ to denote the permutation in $S_{n+m}$ with
$$
(\pi \oplus \sigma)(j) = \begin{cases}
\pi(j) & if\ 1 \leq j \leq n \\ 
n+\sigma(j-n) & if \ n+1 \leq j \leq n+m
\end{cases}
$$
for all $j$, $1\leq j \leq n $, and we write  $\pi \ominus \sigma$ to denote the permutation in $S_{n+m}$ with 
$$
(\pi \ominus \sigma)(j) = \begin{cases}
m + \pi(j) &if\ 1 \leq j \leq n \\
\sigma(j-n) & if\ n+1 \leq j \leq n+m
\end{cases}
$$
for all $j$, $1 \leq j \leq n$.
\end{definition}

\begin{example}
If $\pi = 14235$ and $\sigma = 312$ then $c(\pi) = 52431, \ \pi \oplus \sigma = 14235867$, and $\pi \ominus \sigma = 47568312$.
\end{example}

At times we will be working with permutations of length $0$ or $-1$.  
For these, we use the following convention.

\begin{definition}
We write $\emptyset$ to denote the empty permutation, which is the unique permutation of length $0$, and we write $@$ to denote the \emph{antipermutation}, which is the unique permutation of length $-1$.  We have $c(@) = @$ and $1 \oplus @ = @ \oplus 1 = 1 \ominus @ = @ \ominus 1 = \emptyset.$
\end{definition}

Note that $1$ is a complete Baxter permutation, whose corresponding (reduced) Baxter permutation is also $1$.
Since $1$ is also doubly alternating, $\emptyset$ (its compatible anti-Baxter permutation) is a snow leopard permutation.
Similarly, $\emptyset$ is a complete Baxter permutation, whose corresponding Baxter permutation is also $\emptyset$.
Since $\emptyset$ is also doubly alternating and has length $0$, its compatible anti-Baxter permutation is $@$.
In particular, $@$ is a snow leopard permutation.

\section{Even Threads, Odd Threads, and their Decompositions}
\label{sec:evenodd}

Our starting point, and one of the main results of \cite{Comps}, is the following recursive decomposition of the snow leopard permutations of odd positive length.

\begin{theorem}
\label{thm:slpdecom}
\cite[Thm.~2.20 and Prop.~2.22]{Comps}
For any permutation $\pi$ of odd positive length, $\pi$ is a snow leopard permutation if and only if there exist snow leopard permutations $\pi_1$ and $\pi_2$ of odd length such that $\pi = (1 \oplus c(\pi_1) \oplus 1)\ominus 1 \ominus \pi_2$.  For any permutation $\sigma$ of even length, $\sigma$ is a snow leopard permutation if and only if $\sigma = 1 \oplus c(\sigma_1)$ for some snow leopard permutation $\sigma_1$.
In addition, these decompositions are uniquely determined by $\pi$ and $\sigma$, respectively.
\end{theorem}

We use the term \emph{connector} to refer to the entry in $\pi$ corresponding to the final $1$ in the decomposition of $\pi$.  
From the recursive decomposition of $\pi$ in Theorem \ref{thm:slpdecom}, we see that the connector of $\pi$ will always equal $\pi(1)-1$ if $\pi(1) > 1$.  When $\pi(1)=1$, the permutation will not have a connector.  
Note that when a snow leopard permutation $\pi$ does have a connector, the corresponding entry in $c(\pi)$ will be a left-to-right maximum and a fixed point in $c(\pi)$.  

\begin{example}
Let $\pi = 587694321 = (1 \oplus c(123) \oplus 1) \ominus 1 \ominus 321$.  
Since both $123$ and $321$ are snow leopard permutations, $\pi$ must also be a snow leopard permutation, with connector $\pi(6) = 4$.  
Here, $c(\pi) = 523416789$ so $c(\pi)(6) = 6$ is both a left-to-right maximum and a fixed point in $c(\pi)$.
\end{example}

Note that for snow leopard permutations of lengths one and three we have
$$1 = 1 \oplus c(@) \oplus 1,$$
$$123 = 1 \oplus c(1) \oplus 1,$$
and
$$1 = (1 \oplus c(@) \oplus 1) \ominus 1 \ominus (1 \oplus c(@) \oplus 1),$$
so by using Theorem \ref{thm:slpdecom} inductively we get the following natural ``block decomposition'' for snow leopard permutations of odd positive length.

\begin{corollary}
\label{cor:slpblocks}
For any permutation $\pi$ of odd positive length, $\pi$ is a snow leopard permutation if and only if there is a sequence $\pi_1,\ldots,\pi_k$ of snow leopard permutations of odd length such that
$$\pi = (1 \oplus c(\pi_1) \oplus 1) \ominus 1 \ominus (1 \oplus c(\pi_2) \oplus 1) \ominus 1 \ominus \cdots \ominus 1 \ominus (1 \oplus c(\pi_k) \oplus 1).$$
In addition, the sequence $\pi_1,\ldots, \pi_k$ is uniquely determined by $\pi$.
\end{corollary}

Theorem \ref{thm:slpdecom} tells us that the set of snow leopard permutations of odd length is closed under a complicated operation, but we might hope it is also closed under some simpler operation.
Unfortunately, this set is not closed under $\oplus$ or $\ominus$:  neither $321 \ominus 123 = 654123$ nor $321 \oplus 123 = 321456$ are snow leopard permutations, even though $321$ and $123$ are.
However, we do have the following result.

\begin{theorem}
\label{thm:ominusslp}
Suppose $\pi$ and $\sigma$ are snow leopard permutations of odd length.  Then $\pi \ominus 1 \ominus \sigma$ is also a snow leopard permutation.
\end{theorem}
\begin{proof}
Suppose $\pi$ and $\sigma$ are snow leopard permutations of odd length with $|\pi| = 2n-1$. 
We argue by  induction on $n$. 

When $n = 0$ we have $\pi = @$ and $\pi \ominus 1 \ominus \sigma = @ \ominus 1 \ominus \sigma = \sigma$, which is a snow leopard permutation.
When $n = 1$ we have $\pi = 1$ and 
\begin{eqnarray*}
\pi \ominus 1 \ominus \sigma &=& 1 \ominus 1 \ominus \sigma \\
&=& (1 \oplus @ \oplus 1) \ominus 1 \ominus \sigma,
\end{eqnarray*}
which is a snow leopard permutation by Theorem \ref{thm:slpdecom}.

Now suppose $\pi \ominus 1 \ominus \sigma$ is a snow leopard permutation whenever $|\pi|< 2k-1$ and let $|\pi| = 2k-1$.
Since $\pi$ is a snow leopard permutation of odd positive length, by Theorem \ref{thm:slpdecom} there are snow leopard permutations $\pi_1$ and $\pi_2$ with
\begin{eqnarray*}
\pi \ominus 1 \ominus \sigma &=& ((1 \oplus c(\pi_1) \oplus 1) \ominus 1 \ominus \pi_2) \ominus 1 \ominus \sigma \\
&=& (1 \oplus c(\pi_1) \oplus 1) \ominus 1 \ominus (\pi_2 \ominus 1 \ominus \sigma).
\end{eqnarray*}
Since $|\pi| = 2k-1$ and $|(1\oplus c(\pi_1) \oplus 1) \ominus 1| \geq 2,$ we must have $|\pi_2| <2k-1$. 
Now, since $\pi_2$ and $\sigma$ are snow leopard permutations with $|\pi_2| < 2k-1$, $\pi_2 \ominus 1 \ominus \sigma$ is a snow leopard permutation by induction.  
Thus, $\pi \ominus 1 \ominus \sigma = (1 \oplus c(\pi_1) \oplus 1) \ominus 1 \ominus \pi_3$, where $\pi_3$ is a snow leopard permutation, so $\pi \ominus 1 \ominus \sigma$ is a snow leopard permutation by Theorem \ref{thm:slpdecom}.
\end{proof}

\cite{Comps} show in their Corollary 2.24 that snow leopard permutations preserve parity.  
That is, for any snow leopard permutation $\pi$ and for all $j$ with $1 \leq j \leq |\pi|$, we know that $\pi(j)$ is even if and only if $j$ is even.  
We can separate any parity-preserving permutation into two smaller permutations as follows.

\begin{definition}
For any permutation $\pi$ which preserves parity, we write $\pi^e$ to denote the permutation $\pi$ induces on its even entries and we write $\pi^o$ to denote the permutation $\pi$ induces on its odd entries.  
Note that if $|\pi| = 2n+1$ then $|\pi^e| = n$ and $|\pi^o| = n+1$.
\end{definition}

Although in general we will only consider snow leopard permutations of positive odd length, we include one special case.  
If $\pi = @$, then since $|\pi| = -1$, we must have $|\pi^o| = 0$ and $|\pi^e| = -1$.  
This gives us $\pi^o = \emptyset$ and $\pi^e = @$.

\begin{definition}
We say a permutation $\sigma$ is an \emph{even thread} (resp.~\emph{odd thread}) whenever there is a snow leopard permutation $\pi$ of odd length such that $\pi^e = \sigma$ (resp.~$\pi^o = \sigma$), and we say an even thread $\alpha$ and an odd thread $\beta$ are {\em entangled} whenever there is a snow leopard permutation $\pi$ with $\pi_e = \alpha$ and $\pi^o = \beta$.
We write $ET_n$ (resp.~$OT_n$) to denote the set of even (resp.~odd) threads of length $n$. 
\end{definition}

\begin{example}
The permutation $\pi = 587694321$ is a snow leopard permutation with $\pi^o = 34521$ and $\pi^e = 4321$, so $34521$ is an odd thread, $4321$ is an even thread, and these threads are entangled with each other.
\end{example}

The recursive decomposition for snow leopard permutations that we gave in Theorem \ref{thm:slpdecom} induces a similar decomposition on the odd and even threads.

\begin{theorem}
\label{thm:decomp}
A permutation $\alpha$ is an even thread of nonnegative length if and only if there is an even thread $\alpha_1$ and an odd thread $\beta_1$ with
\begin{equation}
\label{eqn:edecomp}
\alpha = c(\beta_1)  \ominus 1 \ominus \alpha_1.
\end{equation}
Similarly, a permutation $\beta$ is an odd thread of positive length if and only if there is an even thread $\alpha_2$ and an odd thread $\beta_2$ with
\begin{equation}
\label{eqn:odecomp}
\beta = (1 \oplus c(\alpha_2) \oplus 1) \ominus \beta_2.
\end{equation}
\end{theorem}
\begin{proof}
($\Rightarrow$) 
If $\alpha$ is an even thread of length $n$, then by definition there is a snow leopard permutation $\pi$ of length $2n+1$ such that $\pi_e = \alpha$.
Since $\pi$ is a snow leopard permutation of odd positive length, by Theorem \ref{thm:slpdecom} there are snow leopard permutations $\pi_1$ and $\pi_2$ of odd length such that $\pi = (1 \oplus c(\pi_1) \oplus 1) \ominus 1 \ominus \pi_2$. 
We see that $\pi^e = c(\pi_1^o) \ominus 1 \ominus \pi_2^e$.

If $\beta$ is an odd thread of length $n+1$, then by definition there is a snow leopard permutation $\pi$ of length $2n+1$ such that $\pi^o = \beta.$  
Again, we can write $\pi = (1 \oplus c(\pi_1) \oplus 1) \ominus 1 \ominus \pi_2.$  
Here we see that $\pi^o = (1 \oplus c(\pi_1^e) \oplus 1) \ominus \pi_2^o$.

($\Leftarrow$) 
Now suppose $\alpha$ is a permutation of length $n$ and there are permutations $\beta_1 \in OT_{k}$ and $\alpha_1 \in ET_{n-k-1}$ such that $\alpha = c(\beta_1) \ominus 1 \ominus \alpha_1$.  
There must be snow leopard permutations $\pi_1$ and $\pi_2$ of lengths $2k-1$ and $2n-2k-1$, respectively, such that $\pi_1^o = \beta_1$ and $\pi_2^e = \alpha_1$.  
By Theorem~\ref{thm:slpdecom}, $\pi = (1 \oplus c(\pi_1) \oplus 1) \ominus 1 \ominus \pi_2$ is a snow leopard permutation of length $2n+1$.  
Here, $\pi^e = c(\pi_1^o) \ominus 1 \ominus \pi_2^e = c(\beta_1) \ominus 1 \ominus \alpha_1 = \alpha$.

Finally, suppose $\beta$ is a permutation of length $n+1$ and there are permutations $\alpha_2 \in ET_{k}$ and $\beta_2 \in OT_{n-k-1}$ such that $\beta = (1 \oplus c(\alpha_2) \oplus 1) \ominus \beta_2$.  
Since $\alpha_2$ is an even thread of length $k$, there is a snow leopard permutation $\pi_1$ of length $2k+1$ such that $\pi_1^e = \alpha_2$.  
Similarly, since $\beta_2$ is an odd thread of length $n-k-1$, there is a snow leopard permutation $\pi_2$ of length $2n-2k-3$ such that $\pi_2^o = \beta_2$.
By Theorem~\ref{thm:slpdecom}, $\pi = (1 \oplus c(\pi_1) \oplus 1) \ominus 1 \ominus \pi_2$ is a snow leopard permutation of length $2n+1$.  
Here, $\pi^o = (1 \oplus c(\pi_1^e) \oplus 1) \ominus \pi_2^o = (1 \oplus c(\alpha_2) \oplus 1) \ominus \beta_2 = \beta.$
\end{proof}

Note that the decomposition of an odd thread given in Theorem~\ref{thm:decomp} is unique, while the decomposition of an even thread is not.  
That is, for a given odd thread $\beta$ there is a unique even thread $\alpha_1$ and a unique odd thread $\beta_1$ for which \eqref{eqn:odecomp} holds.  
The first entry of $\beta$ corresponds to the first $1$ in the decomposition, the largest entry of $\beta$ corresponds to the second $1$, and then $\alpha_1$ and $\beta_1$ are determined. 
But for a given even thread there can be more than one pair $(\alpha_2, \beta_2)$, consisting of an even thread and an odd thread, for which \eqref{eqn:edecomp} holds.

\begin{example}
The odd threads $4657312$ and $7243561$ can be decomposed as $(1 \oplus c(12) \oplus 1) \ominus 312$ and $(1 \oplus @ \oplus 1) \ominus 243561$, respectively.
\end{example}

\begin{example}
\label{ex:evendecomp}
The even thread $\alpha = 653421$ can be decomposed as $c(\emptyset) \ominus 1 \ominus 53421, c(1) \ominus 1 \ominus 3421$ or $c(12435) \ominus 1 \ominus \emptyset$.
\end{example}

In each of the three decompositions of $\alpha$ in Example~\ref{ex:evendecomp}, we see that the middle $1$ corresponds to a left-to-right maximum and fixed point in $c(\alpha)$.  
It's not difficult to see that this will hold for all even threads.  
In particular, suppose $\alpha$ is an even thread, and $\alpha_1$ and $\beta_1$ are an even thread and an odd thread, respectively, for which $\alpha = c(\beta_1) \ominus 1 \ominus \alpha_1$.
Then $c(\alpha) = \beta_1 \oplus 1 \oplus c(\alpha_1)$, and there are exactly $k = |\beta_1|$ entries of $c(\alpha)$ preceding $c(\alpha)(k+1)$, all of which are less than $c(\alpha)(k+1)$.
It follows that $c(\alpha)(k+1)$ is a left-to-right maximum and fixed point in $c(\alpha)$.  
Since the connector of a snow leopard permutation $\pi$ is both a left-to-right maximum and fixed point in $c(\pi)$, and the $1$ in the decomposition of an even thread corresponds to the connector of a snow leopard permutation, we will refer to an entry of an even thread $\alpha$ as an \emph{eligible connector} whenever it corresponds to an entry of $c(\alpha)$ which is both a left-to-right maximum and a fixed point.

It is worth noting that the converse of the above statement does not hold:  some even threads have eligible connectors which do not arise from a decomposition of the thread.
For example, $\alpha = 354621$ is an even thread with eligible connectors 2 and 1.
If 1 is to correspond to the 1 in a decomposition $c(\beta) \ominus 1 \ominus \alpha_1$, then $c(\beta) = 24351$ and $\beta = 42315$.
However, $42315$ is not an odd thread.  
We will prove that although not every eligible connector in an even thread $\alpha$ corresponds to a $1$ in a decomposition of $\alpha$, the leftmost eligible connector does always correspond to such a $1$.  
We begin with two simple ways of constructing new threads from old threads, which arise from the fact that if $\pi$ and $\sigma$ are snow leopard permutations of odd length then $\pi \ominus 1 \ominus \sigma$ is, too.

\begin{proposition}
\label{prop:ominusoddknot}
Suppose $\beta_1$ and $\beta_2$ are odd threads.
Then $\beta_1 \ominus \beta_2$ is also an odd thread, and the permutations which are entangled with $\beta_1 \ominus \beta_2$ are exactly the permutations of the form $\alpha_1 \ominus 1 \ominus \alpha_2$, where $\alpha_1$ and $\alpha_2$ are entangled with $\beta_1$ and $\beta_2$, respectively.
\end{proposition}
\begin{proof}
If $\beta_1$ and $\beta_2$ are odd threads, then there exist snow leopard permutations $\pi_1$ and $\pi_2$ such that $\pi_1^o = \beta_1$ and $\pi_2^o = \beta_2$.  
From Theorem~\ref{thm:ominusslp}, we know that $\pi = \pi_1 \ominus 1 \ominus \pi_2$ is a snow leopard permutation with $\pi^o = \pi_1^o \ominus \pi_2^o$.  
Thus, $\pi_1^o \ominus \pi_2^o = \beta_1 \ominus \beta_2$ is an odd thread.

To prove the rest of the Proposition, first suppose $\alpha_1$ and $\alpha_2$ are even threads which are entangled with the odd threads $\beta_1$ and $\beta_2$, respectively.
Then there are snow leopard permutations $\pi_1$ and $\pi_2$ such that $\pi_1^e = \alpha_1$, $\pi_1^o = \beta_1$, $\pi_2^e = \alpha_2$, and $\pi_2^o = \beta_2$.
By Theorem \ref{thm:ominusslp}, $\pi = \pi_1 \ominus 1 \ominus \pi_2$ is a snow leopard permutation, and we find $\pi^e = \alpha_1 \ominus 1 \ominus \alpha_2$ and $\pi^o = \beta_1 \ominus \beta_2$.
Therefore, every permutation of the form $\alpha_1 \ominus 1 \ominus \alpha_2$, where $\alpha_1$ and $\alpha_2$ are entangled with $\beta_1$ and $\beta_2$, respectively, is entangled with $\beta_1 \ominus \beta_2$.

To prove the reverse inclusion, suppose $\alpha$ is an even thread entangled with $\beta_1 \ominus \beta_2$, and $\pi$ is a snow leopard permutation with $\pi^e = \alpha$ and $\pi^o = \beta_1 \ominus \beta_2$.
By Corollary \ref{cor:slpblocks}, there are snow leopard permutations $\pi_1, \ldots, \pi_k$ of odd length such that
$$\pi = (1 \oplus c(\pi_1) \oplus 1) \ominus 1 \ominus (1 \oplus c(\pi_2) \oplus 1) \ominus 1 \ominus \cdots \ominus 1 \ominus (1 \oplus c(\pi_k) \oplus 1).$$
The 1s in each summand $1 \oplus c(\pi_j) \oplus 1$ are all in odd positions, so there must be some $m$ for which $\beta_1 = \pi_1^o$ and $\beta_2 = \pi_2^o$, where
$$\pi_1 = (1 \oplus c(\pi_1) \oplus 1) \ominus 1 \ominus \cdots \ominus 1 \ominus (1 \oplus c(\pi_m) \oplus 1)$$
and
$$\pi_2 = (1 \oplus c(\pi_{m+1}) \oplus 1) \ominus 1 \ominus \cdots \ominus 1 \ominus (1 \oplus c(\pi_k) \oplus 1).$$
By Corollary \ref{cor:slpblocks}, $\pi_1$ and $\pi_2$ are snow leopard permutations of odd length, so if $\alpha_1 = \pi_1^e$ and $\alpha_2 = \pi_2^e$, then we have $\alpha = \alpha_1 \ominus 1 \ominus \alpha_2$, and $\alpha_1$ and $\alpha_2$ are entangled with $\beta_1$ and $\beta_2$, respectively.
\end{proof}

It is worth noting that even though $@$ is not an odd thread, it behaves like one in some circumstances.
For instance, if $@$ were an odd thread, then for any odd thread $1 \ominus \beta$ we could use Proposition \ref{prop:ominusoddknot} to conclude that $@ \ominus 1 \ominus \beta = \beta$ is also an odd thread.
And if $@$ were an odd thread then we could use a similar argument to show that if $\beta \ominus 1$ is an odd thread then $\beta$ is also an odd thread.
Even though these arguments fail, the results still hold, as we show next.

\begin{proposition}
\label{prop:ominusremoveones}
For any permutation $\beta$ of nonnegative length, if $1 \ominus \beta$ or $\beta \ominus 1$ is an odd thread then $\beta$ is an odd thread.
\end{proposition}
\begin{proof}
Suppose $1 \ominus \beta$ is an odd thread.
Then by \eqref{eqn:odecomp} there is an even thread $\alpha_2$ and an odd thread $\beta_2$ such that $1 \ominus \beta = (1 \oplus c(\alpha_2) \oplus 1) \ominus \beta_2$, and it's not hard to see that this decomposition is unique.
Comparing the largest elements on each side, we find $\alpha_2 = @$ and $\beta_2 = \beta$, so $\beta$ is an odd thread.

Now suppose $\beta \ominus 1$ is an odd thread;  we argue by induction on $|\beta \ominus 1|$.
If $|\beta \ominus 1| = 1$ then $\beta \ominus 1 = 1$ and $\beta = \emptyset$, and the result holds.
If $|\beta \ominus 1| = 2$ then $\beta \ominus 1 = 21$ and $\beta = 1$, and the result holds.
If $|\beta \ominus 1| \ge 3$ then by \eqref{eqn:odecomp} there is an even thread $\alpha_2$ and an odd thread $\beta_2$ such that $\beta \ominus 1 = (1 \oplus c(\alpha_2) \oplus 1) \ominus \beta_2$.
Since $|1 \oplus c(\alpha_2) \oplus 1| \ge 1$, we must have $|\beta_2| < |\beta \ominus 1|$.
By construction there is a permutation $\beta_1$ of nonnegative length such that $\beta_2 = \beta_1 \ominus 1$, and by induction $\beta_1$ is an odd thread.
Therefore, $(1 \oplus c(\alpha_2) \oplus 1) \ominus \beta_1 = \beta$ is also an odd thread, by Theorem \ref{thm:decomp}.
\end{proof}

We now turn our attention to constructing new even threads from old even threads.

\begin{proposition}
\label{prop:eveneven}
Suppose $\alpha_1$ and $\alpha_2$ are even threads.  
Then $\alpha_1 \ominus 1 \ominus \alpha_2$ is also an even thread, and if $\beta_1$ and $\beta_2$ are odd threads entangled with $\alpha_1$ and $\alpha_2$, respectively, then $\beta_1 \ominus \beta_2$ is an odd thread entangled with $\alpha_1 \ominus 1 \ominus \alpha_2$.
\end{proposition}
\begin{proof}
If $\alpha_1$ and $\alpha_2$ are even threads, then there exist snow leopard permutations $\pi_1$ and $\pi_2$ such that $\pi_1^e = \alpha_1$ and $\pi_2^e = \alpha_2$.  From Theorem~\ref{thm:ominusslp}, we know that $\pi = \pi_1 \ominus 1 \ominus \pi_2$ is a snow leopard permutation, with $\pi^e = \pi_1^e \ominus 1 \ominus \pi_2^e$.  Thus, $\pi_1^e \ominus 1 \ominus \pi_2^e = \alpha_1 \ominus 1 \ominus \alpha_2$ is an even thread.  

To prove the rest of the Proposition, suppose $\beta_1$ and $\beta_2$ are odd threads which are entangled with the even threads $\alpha_1$ and $\alpha_2$, respectively.
Then there are snow leopard permutations $\pi_1$ and $\pi_2$ such that $\pi_1^e = \alpha_1$, $\pi_1^o = \beta_1$, $\pi_2^e = \alpha_2$, and $\pi_2^o = \beta_2$.
By Theorem \ref{thm:ominusslp}, $\pi = \pi_1 \ominus 1 \ominus \pi_2$ is a snow leopard permutation, and we find $\pi^e = \alpha_1 \ominus 1 \ominus \alpha_2$ and $\pi^o = \beta_1 \ominus \beta_2$.
In particular, $\beta_1 \ominus \beta_2$ is an odd thread which is entangled with $\alpha_1 \ominus 1 \ominus \alpha_2$.
\end{proof}

Note that if $\alpha_1$ and $\alpha_2$ are even threads, then there may be an odd thread $\beta$ entangled with $\alpha_1 \ominus 1 \ominus \alpha_2$ which does not have the form $\beta_1 \ominus \beta_2$ for odd threads $\beta_1$ and $\beta_2$ which are entangled with $\alpha_1$ and $\alpha_2$, respectively.
For example, if $\alpha_1 = \emptyset$ and $\alpha_2 = 1$ then $\alpha_1 \ominus 1 \ominus \alpha_2 = 21$, which is entangled with $123$.
But $\alpha_1$ is only entangled with 1, $\alpha_2$ is only entangled with $12$ and $21$, and neither $1 \ominus 12 = 312$ nor $1 \ominus 21 = 321$ is equal to $123$.

The next step in proving that the leftmost eligible connector in an even thread always corresponds to a connector in some snow leopard permutation is to show that if any other eligible connector corresponds to a connector in a snow leopard permutation, then the even thread must begin with its largest element.

\begin{lemma}
\label{lem:startwithn}
Suppose $\alpha$ is an even thread of length $n$, the permutation $\pi$ is a snow leopard permutation with $\pi^e = \alpha$, and the connector of $\pi$ does not correspond with the leftmost eligible connector of $\alpha$.
Then $\alpha(1) = n$.
\end{lemma}
\begin{proof}
Let $\alpha(j)$ be the leftmost eligible connector in $\alpha$.
Since this eligible connector does not correspond to the connector in $\pi$, there must be an eligible connector in $\alpha$ to the right of $\alpha(j)$ which does correspond to the connector in $\pi$.
Let this eligible connector be $\alpha(k)$, and note that $k > j$.
Since $\alpha$ is an even thread, \eqref{eqn:edecomp} there is an even thread $\alpha_1$ and an odd thread $\beta_1$ such that $|\beta_1| = k-1$ and 
$$\alpha = c(\beta_1) \ominus 1 \ominus \alpha_1.$$
In particular, $\alpha(j)$ is among $\beta_1(1), \ldots, \beta_1(k-1)$.
Now by \eqref{eqn:odecomp} there is an even thread $\alpha_2$ and an odd thread $\beta_2$ such that
\begin{equation}
\label{eqn:b1decomp}
\beta_1 = (1 \oplus c(\alpha_2) \oplus 1) \ominus \beta_2.
\end{equation}
We claim that $\beta_2 = \emptyset$.

To prove our claim, first note that $c(\beta_1)(j)$ is an eligible connector in $\alpha$, so it is a left-to-right maximum in $c(\alpha)$, a left-to-right minimum in $\alpha$, and a left-to-right minimum in $c(\beta_1)$.
But $c(\alpha)(j)$ is also a fixed point in $c(\alpha)$, so the entries to its left in $c(\alpha)$ are exactly $1,2,\ldots, j-1$.
This means the entries to the left of $\alpha(j)$ in $\alpha$ are $n, n-1, \ldots, n-j+2$, which are all of the entries of $\alpha$ which are larger than $\alpha(j)$.
Therefore, the entries to the left of $c(\beta_1)(j)$ are $k-1, k-2, \ldots, k-j+1$, the largest entries of $c(\beta_1)$, and $c(\beta_1)(j) = k-j$.
But $c(\beta_1) = (1 \ominus \alpha_2 \ominus 1) \oplus \beta_2$, as in Figure \ref{fig:beta1decomp}.
\begin{figure}[ht]
\centering
\includegraphics[width=0.35\textwidth]{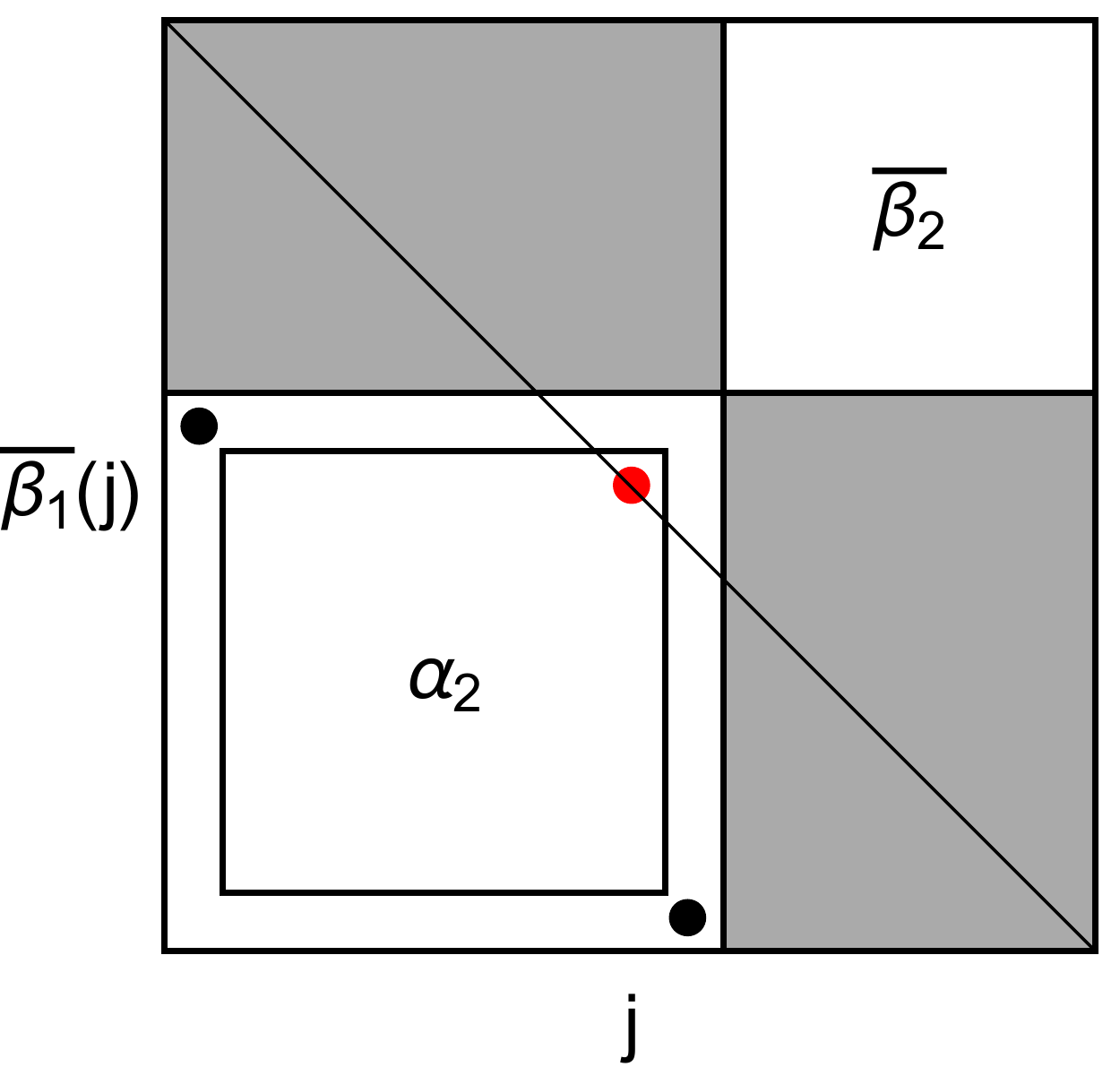}
\caption{A decomposition of $c(\beta_1)$.}
\label{fig:beta1decomp}
\end{figure}
Now if $c(\beta_2) \neq \emptyset$, and $c(\beta_1)(j)$ falls in $1 \ominus \alpha_2 \ominus 1$, then one of $k-1, \ldots, k-j+1$ is to the right of $c(\beta_1)(j)$, which is a contradiction.
On the other hand, if $c(\beta_1)(j)$ falls in $c(\beta_2)$, then there is an entry smaller than $c(\beta_1)(j)$ to the left of $c(\beta_1)(j)$ (since $1 \ominus \alpha_2 \ominus 1$ must be nonempty), which is also a contradiction.
Therefore, $\beta_2 = \emptyset$.

Since $\beta_2 = \emptyset$, by \eqref{eqn:b1decomp} we have $\beta_1 = 1 \oplus c(\alpha_2) \oplus 1$, so $\beta_1(1) = 1$ and $\alpha(1) = n$, as desired.
\end{proof}

Now that we have assembled all of the necessary tools, we are ready to show that the leftmost eligible connector in an even thread always corresponds to a connector in some snow leopard permutation.

\begin{theorem}
\label{thm:ekdecomp}
If $\alpha$ is an even thread of length $n$, then there exists a snow leopard permutation $\pi$ with $\pi^e = \alpha$ whose connector corresponds to the leftmost eligible connector in $\alpha$.  
In particular, if this entry is $c(\alpha)(n-j+1)$, then $\pi(1) = 2j+1$.  
By convention, we say $j = 0$ if no such entry exists.
\end{theorem}
\begin{proof}
Suppose $\alpha$ is an even thread of length $n$, and $\pi$ is a snow leopard permutation with $\pi^e = \alpha$.
If the connector of $\pi$ corresponds with the leftmost eligible connector of $\alpha$, then the result holds.
If not, then by Lemma \ref{lem:startwithn} we have $\alpha(1) = n$.
By \eqref{eqn:edecomp} there is an even thread $\alpha$ and an odd thread $\beta$ such that
\begin{equation}
\label{eqn:alphadecomp}
\alpha = c(\beta_1) \ominus 1 \ominus \alpha_1.
\end{equation}
We claim that $\alpha = 1 \ominus \alpha_4$ for some even thread $\alpha_4$.

Our claim follows from \eqref{eqn:alphadecomp} if $\beta_1 = \emptyset$, so suppose $\beta_1 \neq \emptyset$.
By \eqref{eqn:odecomp} there is an odd thread $\beta_2$ and an even thread $\alpha_2$ with
$$\beta_1 = (1 \oplus c(\alpha_2) \oplus 1) \ominus \beta_2,$$
which implies
$$\alpha = ((1 \ominus \alpha_2 \ominus 1) \oplus c(\beta_2)) \ominus 1 \ominus \alpha_1.$$
Since $\alpha(1) = n$, we must have $\beta_2 = \emptyset$.
Furthermore, $1 \ominus \alpha_1 = \emptyset \ominus 1 \ominus \alpha_1$ is an even thread by Proposition \ref{prop:eveneven};  writing $\alpha_3$ to denote this thread, we have
$$\alpha = 1 \ominus \alpha_2 \ominus 1 \ominus \alpha_3.$$
Using Proposition \ref{prop:eveneven} again, we see that $\alpha_2 \ominus 1 \ominus \alpha_3$ is also an even thread, which we denote by $\alpha_4$, and our claim follows.

To complete the proof, note that since $\alpha_4$ is an even thread, there is a snow leopard permutation $\pi$ with $\pi^e = \alpha_4$.
By Theorem \ref{thm:ominusslp}, $1\ominus 1 \ominus \pi$ is a snow leopard permutation with even thread $\alpha$, whose connector corresponds to $\alpha(1)$, the leftmost eligible connector of $\alpha$.
\end{proof}

Theorem \ref{thm:ekdecomp} tells us that while a given even thread $\alpha$ may have several decompositions as in \eqref{eqn:edecomp}, it always has such a decomposition in which the 1 in \eqref{eqn:edecomp} corresponds to the leftmost eligible connector in $\alpha$.
This decomposition is clearly unique, so we will call it the {\em leftmost decomposition} of $\alpha$.

\section{Even Threads, Odd Threads, and Restricted Catalan Paths}
\label{sec:eopaths}

Now that we have recursive decompositions of the even and odd threads, and we have investigated the extent to which these decompositions are unique for a given thread, it's natural to consider the numbers of these threads of each length.
In Table \ref{table:knotnums}
\begin{table}[ht]
\centering
\begin{tabular}{c|c|c|c|c|c|c|c|c}
$n$ & $-1$ & 0 & 1 & 2 & 3 & 4 & 5 & 6 \\
\hline
$|ET_n|$ & 1 & 1 & 1 & 2 & 6 & 17 & 46 & 128 \\
\hline
$|OT_n|$ & 0 & 1 & 1 & 2 & 4 & 9 & 23 & 63 \\
\end{tabular}
\caption{The number of even threads and odd threads of length six or less.}
\label{table:knotnums}
\end{table}
we have the number of even and odd threads of length six or less.

It turns out that these numbers also count some more familiar combinatorial objects.
To describe these objects, recall that a {\em Catalan path} of length $n$ is a lattice path from $(0,0)$ to $(n,n)$ consisting of unit North $(0,1)$ and East $(1,0)$ steps which does not pass below the line $y = x$.
For convenience, we sometimes write a Catalan path as a sequence of $N$s and $E$s, with $N$ denoting a North step and $E$ denoting an East step.
For instance, the five Catalan paths of length three are $NNNEEE$, $NNENEE$, $NNEENE$, $NENNEE$, and $NENENE$.
If $p$ is a Catalan path of length $n$, then we will write $p^r$ to denote the {\em reverse} of $p$, which is the path obtained by reflecting $p$ over the line $x+y=n$.
In terms of $N$s and $E$s, reversing a Catalan path is equivalent to reversing the corresponding string and then exchanging $N$s and $E$s.
For example, if $p = NENNEE$ then $p^r = NNEENE$.

\cite{DyckPaths} show that the number of Catalan paths of length $n \ge 1$ with $k$ occurrences of $NEEN$ (in consecutive entries) is given by 
$$a_{n,k} = \frac{1}{n} \binom{n}{k} \sum_{j=k}^{\lfloor \frac{n-1}{2} \rfloor}(-1)^{j-k}\binom{n-k}{j-k}\binom{2n-3j}{n-j+1}.$$
Notice that when $k = 0$ the first few terms of this sequence (beginning with $n = 1$) are $1,2,4,9$, and $23$, which suggests that $|OT_n| = a_{n,0}$.
On the other hand, the values of $|ET_n|$ appear to match OEIS sequence A102403, whose $n$th term is the number of Catalan paths of length $n$ with no ascent of length exactly two.  
With this in mind, we introduce some notation for the sets of these paths of a given length.

\begin{definition}
For each $n \geq 0$, we write $ENNE_n$ to denote the set of Catalan paths of length $n$ which have no ascent of length exactly two.  
Similarly, we write $NEEN_n$ to denote the set of Catalan paths of length $n$ which do not contain the four consecutive steps $NEEN$. 
\end{definition}

Our data suggest that $|ET_n| = |ENNE_{n+1}|$ and $|OT_n| = |NEEN_n|$ for each nonnegative integer $n$.
To prove these results, we introduce decompositions of the paths in $ENNE_n$ and $NEEN_n$ which mirror our decompositions of the even and odd threads, respectively.

\begin{theorem}
\label{thm:Andecomp}
Suppose $n$ is a positive integer.
For each $p \in ENNE_n$, there are unique nonnegative integers $k$ and $l$, and unique Catalan paths $a \in ENNE_k$ and $b \in NEEN_l$, such that $n=k+l+1$, $p = a N b^r E$, and $b$ does not end with $NE$.  
Conversely, for every $a \in ENNE_k$ and $b \in NEEN_l$ such that $b$ does not end with $NE$, the path $a N b^r E$ is in $ENNE_n$, where $n=k+l+1$.
\end{theorem}
\begin{proof}
($\Rightarrow$) Suppose $p \in ENNE_n$.  
Since $p$ is a Catalan path which begins at $(0,0)$ and ends at $(n,n)$, it must return to the diagonal $y=x$ at least once.  
Suppose the last time $p$ touches the diagonal before $(n,n)$ is at $(k,k)$.  
Let $a$ denote the subpath of $p$ from $(0,0)$ to $(k,k)$ and let $b$ denote the subpath of $p$ from $(k, k+1)$ to $(n-1, n)$.  
Then $a$ and $b$ are Catalan paths, and $|a| = k$ and $|b| = n-k-1$, so $|a| + |b| = n-1$.  
Furthermore, if $c$ and $d$ are any Catalan paths with $|c| = k$ and $|d| = n-k-1$ then the Catalan path $c N d^r E$ has its last return to the diagonal before $(n,n)$ at $(k,k)$.
Therefore, $a$, $b$, $k$, and $l = n-k-1$ are uniquely determined by $p$, and it remains only to show that $a \in ENNE_k$ and $b^r \in NEEN_{n-k-1}$.

Since $a$ is a subpath of $p$ and $p$ contains no ascent of length exactly two, neither does $a$.   
Thus $a \in ENNE_k$.  

To show $b^r \in NEEN_{n-k-1}$, first notice that since $NbE \in ENNE_{n-k}$, if $b$ contains an ascent of length exactly two then this ascent must start the path.
Equivalently, if $b^r$ contains a plateau of length exactly two, then this plateau must end the path.
Thus, $b^r$ does not contain $NEEN$, so $b^r \in NEEN_{n-k-1}.$
In addition, $b^r$ cannot start with $NE$, for otherwise the $N b$ portion of $p$ would contain an ascent of length exactly two.  
Thus, $b^r$ does not end with an $NE$.

($\Leftarrow$) 
Suppose $a \in ENNE_k$ and $b \in NEEN_l$, and $b$ does not end with $NE$.  
We want to show that the path $a N b^r E$ does not contain an ascent of length exactly two.  
Since $a \in ENNE_k$, if $a N b^r E$ did contain such an ascent, it would have to occur in the $N b^r E$ portion of the path.  
Additionally, $b$ contains no $NEEN$, so $b^r$ contains no $ENNE$.  
Thus, the problematic ascent would have to begin with the $N$ of $N b^r E$.  
However, we assumed that $b$ does not end with $NE$, so $b^r$ does not begin with $NE$.  
Thus, there is no problematic ascent in $a N b^r E$, and this path of length $n=k+l+1$ will be in $ENNE_n$.
\end{proof}

\begin{theorem}
\label{thm:Bndecomp}
Suppose $n$ is a positive integer.
For each $p \in NEEN_n$ there are unique nonnegative integers $k$ and $l$, and unique Catalan paths $a \in ENNE_k$ and $b \in NEEN_l$, such that $n=k+l+1$ and $p = a^r N b E$.
Conversely, for every $a \in ENNE_k$ and $b \in NEEN_l$, the path $a^r N b E$ is in $NEEN_n$, where $n=k+l+1$.
\end{theorem}
\begin{proof}
($\Rightarrow$) 
Suppose $p \in NEEN_n$.  
As in the proof of Theorem~\ref{thm:Andecomp}, let $a$ denote the subpath of $p$ from $(0,0)$ to $(k,k)$ and let $b$ denote the subpath of $p$ from $(k, k+1)$ to $(n-1,1)$, where $(k,k)$ is the point at which $p$ last touches the diagonal before $(n,n)$.  
As before, for any Catalan paths $c$ and $d$ with $|c| = k$ and $|d| = n-k-1$ the Catalan path $c^r N d E$ has its last return to the diagonal at $(k,k)$, so $a$, $b$, $k$, and $l = n-k-1$ are uniquely determined by $p$.

Since $NbE$ is a path in $NEEN_{n-k}$, we have $b \in NEEN_{n-k-1}$.  
Similarly, since $a$ does not contain $NEEN$, the path $a^r$ does not contain $ENNE$.
Therefore, if $a^r$ contains an ascent of length exactly two, then this ascent must come at the beginning of the path, in which case $a^r$ begins with $NNE$.  
In this case, $a$ would end with $NEE$, so $p$ would contain $NEEN$. 
This contradiction shows $a^r \in ENNE_k$.
 
($\Leftarrow$) 
Suppose $a \in ENNE_k$ and $b \in NEEN_l$, and let $p = a^r N b E$. 
We want to show that $p$ does not contain $NEEN$.  
Since there is no $NEEN$ in $b$, if $p$ contains $NEEN$, then it must fall in the $a^r N$ portion of $p$.  
However, $a \in ENNE_k$ so $a$ has no ascent of length exactly two, meaning $a^r$ has no plateau of length exactly two.  
Thus, $a^r N$ also cannot contain $NEEN$, so $p \in NEEN_n$, where $n = |a| + |b| + 1 = k+l+1$.
\end{proof}

Now that we can decompose threads and Catalan paths in similar ways, we can recursively define functions from  $ENNE_{n+1}$ to $ET_n$, and from $NEEN_n$ to $OT_n$;  these functions will turn out to be bijections.

\begin{theorem}
\label{thm:HJ}
For each nonnegative integer $n$, there exist unique functions
$$H : ENNE_{n+1} \rightarrow ET_n$$
$$J : NEEN_n \rightarrow OT_n$$
such that $H(\emptyset) = @$, $J(\emptyset) = \emptyset$, if $p = a N b^r E$ for Catalan paths $a$ and $b$ then 
\begin{equation}
\label{eqn:Hrecurrence}
H(p) = c(J(b)) \ominus 1 \ominus H(a),
\end{equation}
and if $p = a^r N b E$ for Catalan paths $a$ and $b$ then 
\begin{equation}
\label{eqn:Jrecurrence}
J(p) = (1 \oplus c(H(a)) \oplus1) \ominus J(b).
\end{equation}
\end{theorem}
\begin{proof}
The result is clear for $n = 0$, so suppose $n \ge 1$.

Since every Catalan path $p$ of positive length has unique decompositions into the forms $a N b^r E$ and $a^r N b E$, some simple computations with lengths shows that there exist unique functions $H : ENNE_{n+1} \rightarrow S_n$ and $J : NEEN_n \rightarrow S_n$ satisfying \eqref{eqn:Hrecurrence} and \eqref{eqn:Jrecurrence}.
So it's sufficient to show that if $p \in ENNE_{n+1}$ then $H(p) \in ET_n$ and if $p \in NEEN_n$ then $J(p) \in OT_n$.

The result is true by construction for $n = 0$, and it's routine to check that $H(NE) = \emptyset \in ET_0$ and $J(NE) = 1 \in OT_1$, so suppose $n \ge 2$ and the result holds for all paths of length $n-1$ or less.
If $p \in ENNE_n$ then by Theorem \ref{thm:Andecomp} there are unique paths $a \in ENNE_k$ and $b \in NEEN_l$, such that $n=k+l+1$, $p = a N b^r E$, and $b$ does not end with $NE$.
By induction $H(a) \in ET_{k-1}$ and $J(b) \in OT_l$, so by \eqref{eqn:Hrecurrence} and \eqref{eqn:edecomp} the permutation $H(p)$ is in $ET_{n-1}$.
The proof that if $p \in NEEN_n$ then $J(p) \in OT_n$ is similar, using Theorem \ref{thm:Bndecomp}, along with \eqref{eqn:Jrecurrence} and \eqref{eqn:odecomp}.
\end{proof}

In Table \ref{table:HandJvalues}
\begin{table}[ht]
\centering
\begin{tabular}{c|c}
path $a$ & $H(a)$ \\
\hline
$\emptyset$ & $@$ \\
\hline
$NE$ & $\emptyset$ \\
\hline
$NENE$ & $1$ \\
\hline
$NENENE$ & $21$ \\
$NNNEEE$ & $12$ \\
\hline
$NNNNEEEE$ & $123$ \\
$NNNENEEE$ & $132$ \\
$NNNEENEE$ & $213$ \\
$NNNEEENE$ & $312$ \\
$NENNNEEE$ & $231$ \\
$NENENENE$ & $321$ \\
\end{tabular}
\hspace{70pt}
\begin{tabular}{c|c}
path $a$ & $J(a)$ \\
\hline
$\emptyset$ & $\emptyset$ \\
\hline
$NE$ & $1$ \\
\hline
$NENE$ & $12$ \\
$NNEE$ & $21$ \\
\hline
$NNNEEE$ & $321$ \\
$NNENEE$ & $312$ \\
$NENNEE$ & $231$ \\
$NENENE$ & $123$ \\
\end{tabular}
\caption{Values of $H(a)$ and $J(a)$ for small $|a|$.}
\label{table:HandJvalues}
\end{table}
we have the values of $H$ and $J$ on paths of length three or less.
This data suggests $H$ and $J$ are bijections;  to prove this, we construct their inverses.

\begin{theorem}
\label{thm:FG}
For each nonnegative integer $n$, there exist unique functions 
$$F : ET_n \to ENNE_{n+1}$$
$$G : OT_n \to NEEN_n$$
such that $F(@) = \emptyset$, $G(\emptyset) = \emptyset$, if $\pi = c(\beta) \ominus 1 \ominus \alpha$ for an even thread $\alpha$ and an odd thread $\beta$ is the leftmost decomposition of the even thread $\pi$, then
\begin{equation}
\label{eqn:Frecurrence}
F(\pi) = F(\alpha) N G(\beta)^r E,
\end{equation}
and if $\pi = (1 \oplus c(\alpha) \oplus 1) \ominus \beta$ for an even thread $\alpha$ and an odd thread $\beta$ then
\begin{equation}
\label{eqn:Grecurrence}
G(\pi) = F(\alpha)^r N G(\beta) E.
\end{equation}
\end{theorem}
\begin{proof}
This is similar to the proof of Theorem \ref{thm:HJ}.
\end{proof}

\begin{theorem}
\label{thm:inverses}
$F$ and $G$ are the inverse functions of $H$ and $J$, respectively.
\end{theorem}
\begin{proof}
It's routine to check that $H(F(@)) = @$, $H(F(\emptyset)) = \emptyset$, $F(H(\emptyset)) = \emptyset$, $J(G(\emptyset)) = \emptyset$, and $G(J(\emptyset)) = \emptyset$, so the map $H\circ F$ is the identity on $ET_{-1}$ and $ET_0$, the map $F \circ H$ is the identity on $ENNE_0$, the map $J \circ G$ is the identity on $OT_0$, and the map $G \circ J$ is the identity on $NEEN_0$.
Now fix $n \ge 1$, and suppose by induction that for all $k < n$ we have that $H \circ F$ is the identity on $ET_k$, $F \circ H$ is the identity on $ENNE_k$, $J \circ G$ is the identity on $OT_k$, and $G \circ J$ is the identity on $NEEN_k$.
It turns out that the four facts we need to prove are not completely independent, so we start with $J \circ G$.

To show $J \circ G$ is the identity on $OT_n$, suppose we have $\beta \in OT_n$, and let $\beta = (1 \oplus c(\alpha_1) \oplus 1) \ominus \beta_1$ be the decomposition of $\beta$ from \eqref{eqn:odecomp}.
By \eqref{eqn:Grecurrence} we have 
$$G(\beta) = F(\alpha_1)^r N G(\beta_1) E.$$
Now when we decompose $G(\beta)$ in the form $a^r N b E$, we find $a = F(\alpha_1)$ and $b = G(\beta_1)$.
Therefore, by \eqref{eqn:Jrecurrence} we have
$$J(G(\beta)) = (1 \oplus c(H(F(\alpha_1))) \oplus 1) \ominus J(G(\beta_1)).$$
Since $|\beta_1| < n$ and $|\alpha_1| < n$, by induction we have
\begin{eqnarray*}
J(G(\beta)) &=& (1 \oplus c(\alpha_1) \oplus 1) \ominus \beta_1 \\
&=& \beta,
\end{eqnarray*}
as desired.

To show $H \circ F$ is the identity on $ET_n$, suppose we have $\alpha \in ET_n$, and let $\alpha = c(\beta_1) \ominus 1 \ominus \alpha_1$ be the leftmost decomposition of $\alpha$.
By \eqref{eqn:Frecurrence} we have
$$F(\alpha) = F(\alpha_1) N G(\beta_1)^r E.$$
Now when we decompose $F(\alpha)$ in the form $a N b^r E$, we find $a = F(\alpha_1)$ and $b = G(\beta_1)$.
Therefore, by \eqref{eqn:Hrecurrence} we have
$$H(F(\alpha)) = c(J(G(\beta_1))) \ominus 1 \ominus H(F(\alpha_1)).$$
Now if $\alpha_1 = @$ then $|\beta_1| = n$, and by induction and our previous case we have
\begin{eqnarray*}
H(F(\alpha)) &=& c(\beta_1) \ominus 1 \ominus \alpha_1 \\
&=& \alpha,
\end{eqnarray*}
as desired.
On the other hand, if $\alpha_1 \neq @$ then $|\alpha_1| < n$ and $|\beta_1| < n$, and the result follows by induction.

To show $F \circ H$ is the identity on $ENNE_n$, suppose we have $p \in ENNE_n$, and let $p = a N b^r E$ for Catalan paths $a$ and $b$.
By \eqref{eqn:Hrecurrence} we have
\begin{equation}
\label{eqn:Hstar}
H(p) = c(J(b)) \ominus 1 \ominus H(a).
\end{equation}
We claim this is the leftmost decomposition of $H(p)$.
To see this, first note that since $J(b)$ is an odd thread and $H(a)$ is an even thread, there are snow leopard permutations $\pi_1$ and $\pi_2$ such that $J(b) = \pi_1^o$ and $H(a) = \pi_2^e$.
Observe that
$$((1 \oplus c(\pi_1) \oplus 1) \ominus 1 \ominus \pi_2)^e = c(J(b)) \ominus 1 \ominus H(a),$$
and the connector in $(1 \oplus c(\pi_1) \oplus 1) \ominus 1 \ominus \pi_2$ corresponds with the 1 in $c(J(b)) \ominus 1 \ominus H(a)$.
Therefore, by Lemma \ref{lem:startwithn}, if $c(J(b)) \ominus 1 \ominus H(a)$ is not the leftmost decomposition of $H(p)$ then it begins with its largest entry.
This implies that $J(b)$ begins with 1.
Now by \eqref{eqn:Jrecurrence}, if $b = a_1^r N b_1 E$ then $J(b) = \emptyset$, so $b_1 = \emptyset$, and $b$ ends with $NE$.
This means $b^r$ begins with $NE$, and $p = a N b^r E$ has an ascent of length exactly two, contradicting the fact that $p \in ENNE_n$.

Since the right side of \eqref{eqn:Hstar} is the leftmost decomposition of $H(p)$, by \eqref{eqn:Frecurrence} we have
\begin{eqnarray*}
F(H(p)) &=& F(H(a)) N G(J(b))^r E \\
&=& a N b^r E \\
&=& p,
\end{eqnarray*}
by induction.

To show $G \circ J$ is the identity on $NEEN_n$, suppose we have $p \in NEEN_n$, and let $p = a^r N b E$ for Catalan paths $a$ and $b$.
By \eqref{eqn:Jrecurrence} we have
$$J(p) = (1 \oplus c(H(a)) \oplus 1) \ominus J(b).$$
Since $H(a)$ is an even thread, $J(b)$ is an odd thread, and the decomposition of an odd thread in \eqref{eqn:odecomp} is unique, by \eqref{eqn:Grecurrence} we have
$$G(J(p)) = F(H(a))^r N G(J(b)) E.$$
Now the result follows by induction.
\end{proof}

\section{Entangled Threads}
\label{sec:entangled}

For each $n \ge -1$, there is a bipartite graph whose vertices are the even threads of length $n$ and the odd threads of length $n+1$, in which two vertices are adjacent whenever they are entangled.
In Section \ref{sec:eopaths} we gave a partial answer to the question of how many vertices this graph has, by giving bijections between the even threads and the Catalan paths with no ascent of length exactly two, and between the odd threads and the Catalan paths with no $NEEN$.
Nevertheless, there are numerous other questions one might ask about this graph.
In this section we start to answer one of these questions, by characterizing the even threads of length $n-1$ which are entangled with the increasing permutation $12\cdots n$, and by characterizing the odd threads of length $n+1$ which are entangled with the decreasing permutation $n \cdots 2 1$.
Throughout we write $\boxslash_n$ (resp.~$\boxbslash_n$) to denote the increasing (resp.~decreasing) permutation $12\cdots n$ (resp.~$n\cdots 21$).

To state and prove our results, we first need some terminology.
Recall that an {\em involution} $\pi$ is a permutation $\pi$ with $\pi = \pi^{-1}$.
As \cite{G} and \cite{Egge3412} have noted, involutions which avoid the classical pattern $3412$ have a simple recursive structure:  a permutation $\pi$ is a $3412$-avoiding involution if and only if one of the following mutually exclusive conditions holds:  (i)  $\pi = \emptyset$, (ii)  $\pi = 1 \oplus \pi_1$ for a $3412$-avoiding involution $\pi_1$, or (iii) $\pi = (1 \ominus \pi_1 \ominus 1) \oplus \pi_2$ for $3412$-avoiding involutions $\pi_1$ and $\pi_2$.
Moreover, when $\pi$ satisfies (ii) or (iii), the $3412$-avoiding involutions $\pi_1$ and $\pi_2$ are uniquely determined.
Using these results, it is routine to show that the number of $3412$-avoiding involutions of length $n$ is the Motzkin number $M_n$, which may be defined by $M_0 = 1$ and $M_n = M_{n-1} + \sum_{k=2}^n M_{k-2} M_{n-k}$ for $n \ge 1$.
We will encounter another family of objects counted by the Motzkin numbers in Section \ref{sec:janusthreads}, and the interested reader can find still more such families of objects in the paper of \cite{DSMotzkin} and Exercise 6.38 of the book of \cite{StanleyVol2}.
As we show next, the even threads entangled with $\boxslash_n$ and the odd threads entangled with $\boxbslash_n$ are easy to describe in terms of $3412$-avoiding involutions.

\begin{theorem}
\label{thm:oddupevendown}
\begin{enumerate}
\item[{\upshape (i)}]
For any integer $n \ge 1$, the even threads entangled with the odd thread $\boxslash_n$ are the $3412$-avoiding involutions of length $n-1$.
\item[{\upshape (ii)}]
For any integer $n \ge 0$, the odd threads entangled with the even thread  $\boxbslash_n$ are the complements of the $3412$-avoiding involutions of length $n+1$.
\end{enumerate}
\end{theorem}
\begin{proof}
The odd thread $\boxslash_1 = 1$ is only entangled with the even thread $\emptyset$, which is the only $3412$-avoiding involution of length 0.
Similarly, the even thread $\boxbslash_1 = 1$ is entangled with both $12$ and $21$, which are the complements of the $3412$-avoiding involutions of length $2$.
Therefore, the results holds for $n = 1$.
Now fix $n \ge 2$, and suppose the result holds for all $k < n$;  we argue by induction on $n$.

Suppose $\alpha$ is an even thread entangled with $\boxslash_n$.
Then there is a snow leopard permutation $\pi$ such that $\pi^e = \alpha$ and $\pi^o = \boxslash_n$.
By Theorem \ref{thm:slpdecom}, there are snow leopard permutations $\pi_1$ and $\pi_2$ such that $\pi = (1 \oplus c(\pi_1) \oplus 1) \ominus 1 \ominus \pi_2$.
Since $\pi^o = \boxslash_n$, we see that $\pi(1) = 1$, which means $\pi_2 = @$ and $\pi = 1 \oplus c(\pi_1) \oplus 1$.
Now $\pi_1^e = \boxbslash_{n-2}$, so by induction $c(\pi_1^o)$ is a $3412$-avoiding involution.
On the other hand, $\alpha = \pi^e = c(\pi_1^o)$, so $\alpha$ is also a $3412$-avoiding involution.

For the reverse inclusion, suppose $\alpha$ is a $3412$-avoiding involution of length $n-1$.
By induction $c(\alpha)$ is an odd thread entangled with $\boxbslash_{n-2}$, so there is a snow leopard permutation $\pi_1$ with $\pi_1^o = c(\alpha)$ and $\pi_1^e = \boxbslash_{n-2}$.
By Theorem \ref{thm:slpdecom}, the permutation $1 \oplus c(\pi_1) \oplus 1$ is a snow leopard permutation with $\pi^o = \boxslash_n$ and $\pi^e = \alpha$, so $\alpha$ is an even thread entangled with $\boxslash_n$.

Suppose $\beta$ is an odd thread entangled with $\boxbslash_n$.
Then there is a snow leopard permutation $\pi$ such that $\pi^o = \beta$ and $\pi^e = \boxbslash_n$.
By Theorem \ref{thm:slpdecom}, there are snow leopard permutations $\pi_1$ and $\pi_2$ such that $\pi = (1 \oplus c(\pi_1) \oplus 1) \ominus 1 \ominus \pi_2$.
In addition,
$$\beta = \pi^o = (1 \oplus c(\pi_1^e) \oplus 1) \ominus \pi_2^o$$
and
$$\boxbslash_n = \pi^e = c(\pi_1^o) \ominus 1 \ominus \pi_2^e,$$
which implies $\pi_1^o = \boxslash_k$ and $\pi_2^e = \boxbslash_{n-k-1}$ for some $k$ with $-1 \le k \le n-1$.
By induction and a previous case, $\pi_1^e$ and $c(\pi_2^o)$ are $3412$-avoiding involutions, so $c(\beta) = (1 \ominus \pi_1^e \ominus 1) \oplus c(\pi_2^o)$ is also a $3412$-avoiding involution.

For the reverse inclusion, suppose $c(\beta)$ is a $3412$-avoiding involution of length $n+1$.
Then there are $3412$-avoiding involutions $\alpha_1$ of length $k$ and $\beta_1$ of length $n-k-1$, where $-1 \le k \le n-1$, such that
$$c(\beta) = (1 \ominus \alpha_1 \ominus 1) \oplus \beta_1.$$
By induction and a previous case, $\alpha_1$ is an even thread entangled with $\boxslash_{k+1}$ and $c(\beta_1)$ is an odd thread entangled with $\boxbslash_{n-k-2}$.
Therefore, there are snow leopard permutations $\pi_1$ and $\pi_2$ with $\pi_1^e = \alpha_1$, $\pi_1^o = \boxslash_{k+1}$, $\pi_2^o = c(\beta_1)$, and $\pi_2^e = \boxbslash_{n-k-2}$.
Now $\pi = (1 \oplus c(\pi_1) \oplus 1) \ominus 1 \ominus \pi_2$ is a snow leopard permutation by Theorem \ref{thm:slpdecom}, and it has
\begin{eqnarray*}
\pi^o &=& (1 \oplus c(\pi_1^e) \oplus 1) \ominus \pi_2^o \\
&=& (1 \oplus c(\alpha_1) \oplus 1) \ominus c(\beta_1)
\end{eqnarray*}
and
\begin{eqnarray*}
\pi^e &=& c(\pi_1^o) \ominus 1 \ominus \pi_2^e \\
&=& \boxbslash_{k+1} \ominus 1 \ominus \boxbslash_{n-k-2} \\
&=& \boxbslash_n.
\end{eqnarray*}
Since $c(\pi^o) = (1 \ominus \alpha_1 \ominus 1) \oplus \beta_1 = c(\beta)$, we see that $\beta$ is an odd thread entangled with $\boxbslash_n$.
\end{proof}

By combining Theorem \ref{thm:oddupevendown} with Proposition \ref{prop:ominusoddknot}, we can characterize the even threads entangled with a much larger class of odd threads.

\begin{definition}
\label{defn:uplayered}
For $m \ge 1$ and any positive integers $l_1,\ldots, l_m$, we write $\boxslash_{l_1,\ldots, l_m}$ to denote the permutation $\boxslash_{l_1} \ominus \boxslash_{l_2} \ominus \cdots \ominus \boxslash_{l_m}$.
We sometimes call a permutation of the form $\boxslash_{l_1,\ldots, l_m}$ {\em up-layered}.
\end{definition}

\begin{corollary}
\label{cor:evenforlayeredodd}
For $m \ge 1$ and any positive integers $l_1,\ldots, l_m$, the permutation $\boxslash_{l_1,\ldots,l_m}$ is an odd thread, and the even threads entangled with it are the permutations of the form $\pi_1 \ominus 1 \ominus \pi_2 \ominus 1 \ominus \cdots \ominus 1 \ominus \pi_m$, where $\pi_j$ is a $3412$-avoiding involution of length $l_j-1$ for $1 \le j \le m$.
\end{corollary}
\begin{proof}
This follows from Theorem \ref{thm:oddupevendown}(i) and Proposition \ref{prop:ominusoddknot} by induction on $m$.
\end{proof}

\begin{corollary}
\label{cor:layeredoddMotzkinproduct}
For any $m \ge 1$ and any positive integers $l_1,\ldots,l_m$, there are exactly $M_{l_1-1} \cdots M_{l_m-1}$ even threads entangled with the odd thread $\boxslash_{l_1,\ldots,l_m}$.
\end{corollary}
\begin{proof}
This is immediate from Corollary \ref{cor:evenforlayeredodd}, since the number of $3412$-avoiding involutions of length $l_j-1$ is $M_{l_j-1}$ for $1 \le j \le m$.
\end{proof}

It is natural at this point to try to use Theorem \ref{thm:oddupevendown}(ii) and Proposition \ref{prop:eveneven} to prove a analogues of Corollaries \ref{cor:evenforlayeredodd} and \ref{cor:layeredoddMotzkinproduct} for even threads which are layered in the opposite direction.
Unfortunately, doing this only allows us to recover Theorem \ref{thm:oddupevendown}(ii).
Nevertheless, we can use Corollary \ref{cor:evenforlayeredodd} to prove the analogues we seek.
We begin by making the phrase ``layered in the opposite direction'' precise.

\begin{definition}
\label{defn:downlayered}
For $m \ge 1$ and any positive integers $l_1,\ldots, l_m$, we write $\boxbslash_{l_1,\ldots, l_m}$ to denote the permutation $\boxbslash_{l_1} \oplus \boxbslash_{l_2} \oplus \cdots \oplus \boxbslash_{l_m}$.
We sometimes call a permutation of the form $\boxbslash_{l_1,\ldots, l_m}$ {\em down-layered}.
\end{definition}

We now have a natural analogue of Corollary \ref{cor:evenforlayeredodd} for down-layered even threads.

\begin{theorem}
\label{thm:oddforlayeredeven}
For $m \ge 1$ and any positive integers $l_1,\ldots,l_m$, the permutation $\boxbslash_{l_1,\ldots,l_m}$ is an even thread, and the odd threads entangled with it are the permutations of the form $1 \oplus \pi_1 \oplus 1 \oplus \pi_2 \oplus 1 \oplus \cdots \oplus 1 \oplus \pi_m \oplus 1$, where $\pi_j$ is the complement of a $3412$-avoiding involution of length $l_j-1$ for $1 \le j \le m$.
\end{theorem}
\begin{proof}
First note that the result is Theorem \ref{thm:oddupevendown}(ii) when $m = 1$, so we may assume $m \ge 2$.

Since $@$ is an even thread and $\boxslash_{l_1,\ldots,l_m}$ is an odd thread, by Theorem \ref{thm:decomp} we see that $\boxbslash_{l_1,\ldots,l_m} = c(\boxslash_{l_1,\ldots,l_m} \ominus 1 \ominus @$ is an even thread.

To prove the rest of the result, suppose $\pi$ is a snow leopard permutation with $\pi^e = \boxbslash_{l_1,\ldots,l_m}$.
By Theorem \ref{thm:slpdecom}, there are snow leopard permutations $\sigma_1$ and $\sigma_2$ such that $\pi = (1 \ominus c(\sigma_1)\ominus 1) \ominus 1 \ominus \sigma_2$.
If $\sigma_2 \neq @$ then there is an odd thread $\beta_1$ and a nonempty even thread $\alpha_1$ such that $\pi^e = c(\beta_1) \ominus 1 \ominus \alpha_1$.
In $c(\beta_1) \ominus 1 \ominus \alpha_1$, the largest entry appears to the left of the smallest entry, but in $\pi^e = \boxbslash_{l_1,\ldots,l_m}$ the largest entry appears to the right of the smallest entry (since $m \ge 2$).
Therefore, $\sigma_2 = @$ and $\pi = 1 \oplus c(\sigma_1) \oplus 1$.

Since $\pi^e = \boxbslash_{l_1,\ldots,l_m}$, we have $\sigma_1^o = c(\boxbslash_{l_1,\ldots,l_m}) = \boxslash_{l_1,\ldots,l_m}$.
Now by Corollary \ref{cor:evenforlayeredodd}, the permutation $\sigma_1^e$ has the form $\pi_1 \ominus 1 \ominus \cdots \ominus 1 \ominus \pi_{l_m}$, where $\pi_j$ is a $3412$-avoiding involution of length $l_j-1$ for $1 \le j \le m$.
This means $\pi^o = 1 \oplus c(\pi_1) \oplus 1 \oplus \cdots \oplus 1 \oplus c(\pi_{l_m}) \oplus 1$, and the result follows.
\end{proof}

\begin{corollary}
\label{cor:layeredevenMotzkinproduct}
For $m \ge 1$ and any positive integers $l_1,\ldots,l_m$, there are exactly $M_{l_1-1} \cdots M_{l_m-1}$ odd threads entangled with the even thread $\boxbslash_{l_1,\ldots,l_m}$.
\end{corollary}
\begin{proof}
This is immediate from Theorem \ref{thm:oddforlayeredeven}, since the number of $3412$-avoiding involutions of length $l_j-1$ is $M_{l_j-1}$ for $1 \le j \le m$.
\end{proof}

Inspired by Corollaries \ref{cor:layeredoddMotzkinproduct} and \ref{cor:layeredevenMotzkinproduct}, along with some numerical data, we make the following conjecture.

\begin{conjecture}
\label{conj:Motzkinproduct}
For any even (resp.~odd) thread $\alpha$ (resp.~$\beta$), the number of odd (resp.~even) threads entangled with $\alpha$ (resp.~$\beta$) is a product of Motzkin numbers.
\end{conjecture}

We have verified Conjecture \ref{conj:Motzkinproduct} for all even threads of length eleven or less and all odd threads of length twelve or less.

\section{Janus Threads}
\label{sec:janusthreads}

Having investigated the entanglement relation between even and odd threads, we now turn our attention to those threads which are both even and odd.

\begin{definition}
We say a permutation is a {\em Janus thread} whenever it is both an even thread and an odd thread.
For convenience, we also regard the antipermutation $@$ as a Janus thread (of length $-1$).
For each $n \ge -1$, we write $JT_n$ to denote the set of Janus threads of length $n$.
\end{definition}

Remarkably, for small $n$ nearly every odd thread is also a Janus thread:  the smallest odd thread that is not also an even thread is $3412$.
Nevertheless, it's natural to ask how many Janus threads of length $n$ there are.
In Table \ref{table:janusthreadcount} we have the number of Janus threads of length nine or less.
\begin{table}[ht]
\centering
\begin{tabular}{c|c|c|c|c|c|c|c|c|c|c|c}
$n$ & $-1$ & 0 & 1 & 2 & 3 & 4 & 5 & 6 & 7 & 8 & 9 \\
\hline
$|JT_n|$ & 1 & 1 & 1 & 2 & 4 & 8 & 17 & 37 & 82 & 185 & 423 
\end{tabular}
\caption{The number of Janus threads of length nine or less.}
\label{table:janusthreadcount}
\end{table}

As was the case for even and odd threads, Janus threads are related to a certain kind of lattice path.
To describe these lattice paths, first recall that a {\em Motzkin path} of length $n$ is a lattice path from $(0,0)$ to $(n,0)$ consisting of unit Up $(1,1)$, Down $(1,-1)$, and Level $(1,0)$ steps which does not pass below the $x$-axis.
As their name suggests, the Motzkin paths of length $n$ are counted by the Motzkin number $M_n$, which we met in Section \ref{sec:entangled}. 
A {\em peak} in a Motzkin path is a pair of consecutive steps in which the first step is an Up step and the second is a Down step.
We say a Motzkin path is {\em peakless} whenever it has no peaks, and we write $UD_n$ to denote the set of peakless Motzkin paths of length $n$.
For example, $UD_5$ consists of eight Motzkin paths:  $LLLLL$, $LLULD$, $LULDL$, $ULDLL$, $LULLD$, $ULLDL$, $ULLLD$, and $ULLLD$.
In Table \ref{table:peaklessmotzkin} we have the number of peakless Motzkin paths of length ten or less.
\begin{table}[ht]
\centering
\begin{tabular}{c|c|c|c|c|c|c|c|c|c|c|c}
$n$ & 0 & 1 & 2 & 3 & 4 & 5 & 6 & 7 & 8 & 9 & 10 \\
\hline
$|UD_n|$ & 1 & 1 & 1 & 2 & 4 & 8 & 17 & 37 & 82 & 185 & 423
\end{tabular}
\caption{The number of peakless Motzkin paths of length ten or less.}
\label{table:peaklessmotzkin}
\end{table}
These paths are counted by a certain sequence of generalized Catalan numbers, which is sequence A004148 in the OEIS;  its terms satisfy $a_n = a_{n-1}  + \sum_{k=1}^{n-2} a_k a_{n-2-k}$ for $n \ge 0$.

As a comparison of Tables \ref{table:janusthreadcount} and \ref{table:peaklessmotzkin} suggests, Janus threads of length $n$ are in bijection with peakless Motzkin paths of length $n+1$.
As a first step in constructing a bijection between these two sets, we describe how to construct peakless Motzkin paths recursively.

\begin{theorem}
\label{thm:peaklessdecomp}
Suppose $n$ is a positive integer.
For each $p \in UD_n$, exactly one of (i) and (ii) below holds.
\begin{enumerate}
\item[{\upshape (i)}]
There is a unique Motzkin path $a \in UD_{n-1}$ such that $p = L a$.
\item[{\upshape (ii)}]
There are unique integers $k \ge 1$ and $l \ge 0$ and unique Motzkin paths $a \in UD_k$ and $b \in UD_l$ such that $n = k+l+2$ and $p = U a D b$.
\end{enumerate}
Conversely, both of the following hold.
\begin{enumerate}
\item[{\upshape (iii)}]
For every $a \in UD_{n-1}$, we have $L a \in UD_n$.
\item[{\upshape (iv)}]
For any integers $k \ge 1$ and $l \ge 0$ with $n = k+l+2$, and any $a \in UD_k$ and $b \in UD_l$, we have $U a D b \in UD_n$.
\end{enumerate}
\end{theorem}
\begin{proof}
This is similar to the proof of Theorem \ref{thm:Andecomp}.
\end{proof}

We can also use our recursive decompositions of the even and odd threads in Theorem \ref{thm:decomp} to describe how to construct Janus threads recursively.

\begin{theorem}
\label{thm:januslargestfirst}
Suppose $\gamma$ is a permutation of length at least one which begins with its largest entry.
Then $\gamma$ is a Janus thread if and only if there is a Janus thread $\gamma_1$ such that $\gamma = 1 \ominus \gamma_1$.
Moreover, when these conditions hold, $\gamma_1$ is determined by $\gamma$.
\end{theorem}
\begin{proof}
($\Rightarrow$)
Suppose $\gamma$ is a Janus thread which begins with its largest entry.
Since $\gamma$ begins with its largest entry, it has the form $1 \ominus \gamma_1$ for a unique permutation $\gamma_1$.
In particular, the last statement of the theorem holds.

To show $\gamma_1$ is an odd thread, first note that since $\gamma$ is an odd thread, by \eqref{eqn:odecomp} there is an even thread $\alpha_1$ and an odd thread $\beta_1$ such that $\gamma = (1 \oplus c(\alpha_1) \oplus 1) \ominus \beta_1$.
Since $\gamma$ begins with its largest entry, we must have $\alpha_1 = @$ and $\beta_1 = \gamma_1$, so $\gamma_1$ is an odd thread.

To show $\gamma_1$ is an even thread, and therefore a Janus thread, first note that since $\gamma$ is an even thread, by \eqref{eqn:edecomp} and Theorem \ref{thm:ekdecomp} there is an even thread $\alpha_2$ and an odd thread $\beta_2$ such that $\gamma = c(\beta_2) \ominus 1 \ominus \alpha_2$, and the $1$ corresponds to the leftmost eligible connector in $\gamma$.
If $c(\beta_2)$ begins with its largest entry, then $\beta_2$ begins with 1, and the first entry of $\gamma$ is an eligible connector.
However, this contradicts the fact that the leftmost eligible connector in $\gamma$ is not in $\beta_2$.
Therefore, $c(\beta_2)$ cannot begin with its largest entry.
Since $\gamma = c(\beta_2) \ominus 1 \ominus \alpha_2$ does begin with its largest entry, we must have $\beta_2 = @$ or $\beta_2 = \emptyset$.
The first of these contradicts the fact that the leftmost eligible entry of $\gamma$ does not occur in $\alpha_2$, so we must have $\gamma = 1 \ominus \alpha_2$.
Therefore, $\gamma_1 = \alpha_2$ is an even thread.

($\Leftarrow$)
This is immediate from Propositions \ref{prop:ominusoddknot} and \ref{prop:eveneven}, the fact that $1$ is an odd thread, and the fact that $\emptyset$ is an even thread. 
\end{proof}

\begin{theorem}
\label{thm:januslargestnotfirst}
Suppose $\gamma$ is a permutation of length at least two which does not begin with its largest entry.
Then $\gamma$ is a Janus thread if and only if there are Janus threads $\gamma_1$ and $\gamma_2$ such that $\gamma_1$ has nonnegative length and $\gamma = (1 \oplus c(\gamma_1) \oplus 1) \ominus 1 \ominus \gamma_2$.
Moreover, when these conditions hold, $\gamma_1$ and $\gamma_2$ are determined by $\gamma$.
\end{theorem}
\begin{proof}
($\Rightarrow$)
Suppose $\gamma$ is a Janus thread which does not begin with its largest entry.
Since $\gamma$ is an odd thread, by \eqref{eqn:odecomp} there is an even thread $\alpha_1$ and an odd thread $\beta_1$ such that $\gamma = (1 \oplus c(\alpha_1) \oplus 1) \ominus \beta_1$.
The second 1 in this decomposition corresponds to the largest entry of $\gamma$, so this decomposition is unique, and the last statement of the theorem holds.
Moreover, since $\gamma$ does not begin with its largest entry, $\alpha_1$ has nonnegative length.
Since all odd threads have nonnegative length, $\beta_1$ also has nonnegative length.

We claim $\beta_1$ begins with its largest entry.
To prove this, suppose by way of contradiction that $\beta_1$ does not begin with its largest entry.
Then by \eqref{eqn:odecomp} there is an even thread $\alpha_2$ of nonnegative length and an odd thread $\beta_2$ such that $\beta_1 = (1 \oplus c(\alpha_2) \oplus 1) \ominus \beta_2$.
Repeating this argument as long as $\beta_j$ does not begin with its largest entry, we find there are even threads $\alpha_1, \ldots,\alpha_k$ of nonnegative length and an odd thread $\beta_{k+1}$ such that
\begin{equation}
\label{eqn:gammadecomp}
\gamma = (1 \oplus c(\alpha_1) \oplus 1 ) \ominus (1 \oplus c(\alpha_2) \oplus 1) \ominus \cdots \ominus (1 \oplus c(\alpha_k) \oplus 1) \ominus 1 \ominus \beta_{k+1}.
\end{equation}
Furthermore, because $\gamma$ is an even thread, $c(\gamma)$ must have a fixed point which is a left-to-right maximum.
This cannot occur in any of the summands $1 \oplus c(\alpha_j) \oplus 1$, so the leftmost eligible connector in $\gamma$ must correspond to the rightmost 1 in the decomposition in \eqref{eqn:gammadecomp}.
By Theorem \ref{thm:ekdecomp}, the permutation $(1 \oplus c(\alpha_1) \oplus 1 ) \ominus (1 \oplus c(\alpha_2) \oplus 1) \ominus \cdots \ominus (1 \oplus c(\alpha_k) \oplus 1)$ is the complement of an odd thread and $\beta_{k+1}$ an even thread.
Now it follows that $(1 \ominus \alpha_1 \ominus 1) \oplus \cdots \oplus (1 \ominus \alpha_k \ominus 1)$ is an odd thread, and by \eqref{eqn:odecomp} there is an even thread $\alpha$ and an odd thread $\beta$ such that $(1 \ominus \alpha_1 \ominus 1) \oplus \cdots \oplus (1 \ominus \alpha_k \ominus 1) = (1 \oplus c(\alpha) \oplus 1) \ominus \beta$.
Matching largest entries on each side of this equation, we find $\beta = \alpha_k \ominus 1$.
But all of the entries of $\beta$ in $(1 \oplus c(\alpha) \oplus 1) \ominus \beta$ are less than every entry to the left of the largest entry, while all of the entries of $\alpha_k \ominus 1$ in $(1 \ominus \alpha_1 \ominus 1) \oplus \cdots \oplus (1 \ominus \alpha_k \ominus 1)$ are greater than all of the entries to the left of the largest entry.
This contradicts the facts that $k \ge 2$ and $\alpha_1 \neq @$.

Since $\beta_1$ begins with its largest entry, we now know that we have $\gamma = (1 \oplus c(\gamma_1) \oplus 1) \ominus 1 \ominus \gamma_2$, where $\gamma_1$ has nonnegative length, $\gamma_1$ is an even thread, $1 \ominus \gamma_2$ is an odd thread, and $1 \oplus c(\gamma_1) \oplus 1$ is the complement of an odd thread.
Furthermore, the rightmost 1 in this decomposition is the leftmost eligible connector in $\gamma$, so by Theorem \ref{thm:ekdecomp} we see that $\gamma_2$ is an even thread.
On the other hand, since $1 \ominus \gamma_2$ is an odd thread, by Proposition \ref{prop:ominusremoveones}, we have that $\gamma_2$ is also an odd thread.
Therefore, $\gamma_2$ is a Janus thread.

To see that $\gamma_1$ is an odd thread, first note that $1 \oplus c(\gamma_1) \oplus 1$ is the complement of an odd thread means $1 \ominus \gamma_1 \ominus 1$ is an odd thread.
Now the result follows from Proposition \ref{prop:ominusremoveones}.

($\Leftarrow$)
Suppose $\gamma_1$ and $\gamma_2$ are Janus threads.
The fact that $(1 \oplus c(\gamma_1) \oplus 1) \ominus 1 \ominus \gamma_2$ is an odd thread is immediate from Proposition \ref{prop:ominusoddknot}, the fact that 1 is an odd thread, and \eqref{eqn:odecomp}.
Similarly, since $\gamma_1$ is an odd thread, $1 \ominus \gamma_1 \ominus 1$ is also an odd thread by Proposition \ref{prop:ominusoddknot}.
Now the fact that $(1 \oplus c(\gamma_1) \oplus 1) \ominus 1 \ominus \gamma_2$ is an even thread follows from \eqref{eqn:edecomp}, since $\gamma_2$ is an even thread.
\end{proof}

The fact that our Janus thread decompositions exactly match our peakless Motzkin path decompositions allows us to construct a recursive bijection between these two sets.

\begin{theorem}
\label{thm:janusrecursivemap}
For each integer $n \ge -1$, there is a unique bijection 
$$K : JT_n \rightarrow UD_{n+1}$$
such that $K(@) = \emptyset$, $K(\emptyset) = L$, if $\gamma = 1 \ominus \gamma_1$ for a Janus thread $\gamma$ then $K(\gamma) = L K(\gamma_1)$ and if $\gamma = (1 \oplus c(\gamma_1) \oplus 1) \ominus 1 \ominus \gamma_2$ for Janus threads $\gamma_1$ and $\gamma_2$ then $K(\gamma) = U K(\gamma_1) D K(\gamma_2)$.
\end{theorem}
\begin{proof}
This is similar to the proof of Theorems \ref{thm:HJ} and \ref{thm:inverses}, using Theorems \ref{thm:peaklessdecomp}, \ref{thm:januslargestfirst}, and \ref{thm:januslargestnotfirst}.
\end{proof}

Although our description of $K$ above is recursive, it turns out this map has an elegant direct combinatorial description as well.
We close the section with this alternative description of $K$.

\begin{definition}
\label{defn:calK}
For any Janus thread of length $n \ge -1$, we define the lattice path ${\cal K}(\gamma)$ as follows.
For $n = -1$ or $n = 0$ we have ${\cal K}(@) = \emptyset$ and ${\cal K}(\emptyset) = L$.
If $n \ge 1$, then we obtain ${\cal K}(\gamma)$ as follows.
\begin{enumerate}
\item
Write $n+1$, followed by $\gamma$, followed by 0.
\item
For each pair of consecutive entries in this new sequence, if the entries are consecutive integers (in either order) then write $L$ between them.
If the consecutive entries are not consecutive integers, then write $U$ between them if they form an ascent, and $D$ if they form a descent.
\item
In the resulting sequence of $n+1$ $U$s, $L$s, and $D$s, number the subsequence of $U$s and $D$s from left to right, beginning with 1.
Change every odd-numbered $U$ to a $D$ and every odd-numbered $D$ to a $U$.
\end{enumerate}
The resulting sequence of $U$s, $L$s, and $D$s is the lattice path ${\cal K}(\gamma)$.
\end{definition}

\begin{example}
\label{ex:calK}
When $\gamma = 576894312$ we write $10\ 576894312\ 0$, and we obtain the sequence $DULULDLDLD$.
Changing the odd-numbered $U$s to $D$s and the odd-numbered $D$s to $U$s, we find that ${\cal K}(578694312) = UULDLDLULD$.
\end{example}

In Example \ref{ex:calK} we obtain a peakless Motzkin path, and in fact it's not difficult to check that ${\cal K}(576894312) = K(576894312)$.
As we show in our final result, ${\cal K}(\gamma)$ is a peakless Motzkin path for every Janus thread $\gamma$, and ${\cal K}(\gamma) = K(\gamma)$.

\begin{theorem}
\label{thm:KiscalK}
For any Janus thread $\gamma$, we have ${\cal K}(\gamma) = K(\gamma)$.
\end{theorem}
\begin{proof}
The result is easy to check when $\gamma$ has length less than two, so suppose $|\gamma| \ge 3$;  we argue by induction on the length of $\gamma$.

If $\gamma$ begins with its largest entry, then by Theorem \ref{thm:januslargestfirst} there is a Janus thread $\gamma_1$ for which $\gamma = 1 \ominus \gamma_1$.
Since $|\gamma_1| = |\gamma|-1 \ge 2$, we have ${\cal K}(\gamma) = L {\cal K}(\gamma_1)$.
By induction, ${\cal K}(\gamma_1) = K(\gamma_1)$, so ${\cal K} (\gamma) = K(\gamma)$ by the definition of $K$.

If $\gamma$ does not begin with its largest entry, then by Theorem \ref{thm:januslargestnotfirst} there are Janus threads $\gamma_1$ and $\gamma_2$ for which $\gamma = (1 \oplus c(\gamma_1) \oplus 1) \ominus 1 \ominus \gamma_2$, where $\gamma_1$ has nonnegative length.
Examining ${\cal K}( (1 \oplus c(\gamma_1) \oplus 1) \ominus 1 \ominus \gamma_2)$, we find ${\cal K}(\gamma) = U {\cal K}(\gamma_1)D{\cal K}(\gamma_2)$.
Now the result follows by induction and the definition of $K$.
\end{proof}

\section{Future Directions}

This work originated in a larger effort to give a nonrecursive characterization of the snow leopard permutations.
For instance, we have made extensive use of the fact that snow leopard permutations preserve parity, but most permutations which preserve parity are not snow leopard permutations.
Similarly, we noted in the Introduction that snow leopard permutations are anti-Baxter permutations, so they avoid the vincular patterns $3\underline{41}2$ and $2\underline{14}3$.
Still, most anti-Baxter permutations which preserve parity are not snow leopard permutations.
\cite{Comps} introduce in their Definition 3.2 a function $\kappa$ which maps each permutation to a lattice path consisting of North and East steps, and they show that $\kappa$ is a bijection between the set of snow leopard permutations of length $2n-1$ and the set of Catalan paths from $(0,0)$ to $(n,n)$.
However, there are many other permutations $\pi$ for which $\kappa(\pi)$ is a Catalan path.
Indeed, let's call a permutation {\em SLP-like} whenever it is an anti-Baxter permutation which preserves parity and maps to a Catalan path under $\kappa$.
Then in Table \ref{table:SLPlike} 
\begin{table}[ht]
\centering
\begin{tabular}{c|c|c|c|c|c}
length & $1$ & $3$ & $5$ & $7$ & $9$ \\
\hline
number of SLP-like permutations & 1 & 2 & 7 & 32 & 175 \\
number of snow leopard permutations & 1 & 2 & 5 & 14 & 42
\end{tabular}
\caption{The number of SLP-like permutations compared with the number of snow leopard permutations.}
\label{table:SLPlike}
\end{table}
we see that most SLP-like permutations are still not snow leopard permutations.
Giving a nonrecursive characterization of the snow leopard permutations remains an open problem.
Similarly, it remains an open problem to prove our Conjecture \ref{conj:Motzkinproduct}.
Finally, as far as we know no one has investigated the permutations which are compatible with alternating Baxter permutations, even though \cite{Cori} have shown that the number of alternating Baxter permutations of length $2n$ (resp.~$2n+1$) is $C_n^2$ (resp.~$C_n C_{n+1}$).
We conjecture that the compatibility relation is a bijection in this case, just as it is in the case of doubly alternating permutations and snow leopard permutations.

\bibliographystyle{abbrvnat}
% use the following instead if you encounter problems 
%\bibliographystyle{alpha}
\bibliography{references}
\label{sec:biblio}

\end{document}